\documentclass[11pt]{article}

\input diagram.tex
\usepackage{amsmath} 
\usepackage{amssymb} 
\usepackage{theorem}
\usepackage{rotating}
\theoremstyle{change} 
\newtheorem{theorem}{Theorem}[section] 
\newtheorem{lemma}[theorem]{Lemma} 
\newtheorem{proposition}[theorem]{Proposition}
\newtheorem{corollary}[theorem]{Corollary}
\theorembodyfont{\rmfamily} 
\newtheorem{remark}[theorem]{Remark}
\newtheorem{example}[theorem]{Example}

\newtheorem{nothing}[theorem]{} 

\newenvironment{proof}{\noindent{\bf Proof}\ }{\qed\bigskip}

\renewcommand{\le}{\leqslant}

\renewcommand{\marginpar}[1]{}

\newcommand{\alphabar}{\overline{\alpha}}
\newcommand{\Aut}{\mathrm{Aut}}
\newcommand{\calA}{\mathcal{A}}
\newcommand{\calB}{\mathcal{B}}
\newcommand{\calE}{\mathcal{E}}
\newcommand{\calF}{\mathcal{F}}
\newcommand{\calFbar}{\overline{\calF}}
\newcommand{\CC}{\mathbb{C}}
\newcommand{\cdotG}{\cdot_G}
\newcommand{\charac}{\mathrm{char}}
\newcommand{\defl}{\mathrm{def}}
\newcommand{\End}{\mathrm{End}}
\newcommand{\Fbar}{\overline{F}}
\newcommand{\FF}{\mathbb{F}}
\newcommand{\GL}{\mathrm{GL}}
\newcommand{\Hom}{\mathrm{Hom}}
\newcommand{\Hombar}{\overline{\Hom}}
\newcommand{\id}{\mathrm{id}}
\newcommand{\im}{\mathrm{im}}

\newcommand{\ind}{\mathrm{ind}}
\newcommand{\infl}{\mathrm{inf}}
\newcommand{\Inj}{\mathrm{Inj}}
\newcommand{\Injbar}{\overline{\Inj}}
\newcommand{\Inn}{\mathrm{Inn}}
\newcommand{\Irr}{\mathrm{Irr}}
\newcommand{\isom}{\mathrm{iso}}

\newcommand{\Lbar}{\overline{L}}
\newcommand{\lexp}[2]{\setbox0=\hbox{$#2$} \setbox1=\vbox to
                 \ht0{}\,\box1^{#1}\!#2}
\newcommand{\liso}{\buildrel\sim\over\longrightarrow}
\newcommand{\myiso}{\buildrel\sim\over\to}
\newcommand{\omegabar}{\overline{\omega}}

\newcommand{\qed}{\nobreak\hfill
                   \vbox{\hrule\hbox{\vrule\hbox to 5pt
                   {\vbox to 8pt{\vfil}\hfil}\vrule}\hrule}}
\newcommand{\QQ}{\mathbb{Q}}
\newcommand{\res}{\mathrm{res}}
\newcommand{\rk}{\mathrm{rk}}
\newcommand{\Out}{\mathrm{Out}}
\newcommand{\stab}{\mathrm{stab}}
\newcommand{\Sigmahat}{\widehat{\Sigma}}
\newcommand{\Sigmatilde}{\widetilde{\Sigma}}
\newcommand{\sigmatilde}{\widetilde{\sigma}}
\newcommand{\St}{\mathrm{St}}
\newcommand{\Thetahat}{\widehat{\Theta}}
\newcommand{\Thetatilde}{\widetilde{\Theta}}
\newcommand{\tr}{\mathrm{tr}}
\newcommand{\trl}{\lhd}
\newcommand{\Utilde}{\widetilde{U}}
\newcommand{\varepsilonhat}{\widehat{\varepsilon}}
\newcommand{\vbar}{\overline{v}}
\newcommand{\Vbar}{\overline{V}}
\newcommand{\Xihat}{\widehat{\Xi}}
\newcommand{\Xitilde}{\widetilde{\Xi}}
\newcommand{\ZZ}{\mathbb{Z}}

\title{Central idempotents of the bifree and left-free double Burnside ring\footnote{{\bf MR Subject Classification:} 19A22, 20C15. 
{\bf Keywords:} Double Burnside ring, ghost ring, primitive central idempotents, fusion systems.}}
\author{\small Robert Boltje\\
  \small Department of Mathematics\\
  \small University of California\\
  \small Santa Cruz, CA 95064\\
  \small U.S.A.\\
  \small boltje@ucsc.edu
  \and
  \small Burkhard K\"ulshammer\\
  \small Mathematisches Institut\\
  \small Friedrich-Schiller-Universit\"at Jena\\
  \small D-07737 Jena\\
  \small Germany\\
  \small kuelshammer@uni-jena.de}
\date{December 2, 2013}

\begin{document}

\sloppy

\maketitle


\begin{abstract}
\noindent 
We determine the blocks, i.e., the primitive central idempotents, of the bifree double Burnside ring and the left-free double Burnside ring, as well as the primitive central idempotents of the algebras arising from scalar extension to $\QQ$.
\end{abstract}


\section{Introduction}\label{sec intro}
The aim of this paper is to find the primitive central idempotents of the subrings $B^\Delta(G,G)$ and $B^{\trl}(G,G)$ of the double Burnside ring $B(G,G)$ of a finite group $G$, as well as the primitive central idempotents of the algebras $\QQ B^\Delta(G,G)$ and $\QQ B^{\trl}(G,G)$. Recall that the double Burnside ring $B(G,G)$ is the Grothendieck group of the category of finite $(G,G)$-bisets with respect to disjoint unions, equipped with the multiplication induced by tensoring $(G,G)$-bisets over $G$. The subrings $B^\Delta(G,G)\subseteq B^{\trl}(G,G)\subseteq B(G,G)$ arise from considering bifree $(G,G)$-bisets and left-free $(G,G)$-bisets. These classes of bisets are of particular interest, since, via the theory of biset functors introduced by Bouc, they are related to globally defined Mackey functors and to globally defined Mackey functors with inflation (or deflation) maps as extra structures. The bifree subring of a $p$-group $S$ is also related to fusion systems on $S$, cf.~\cite{RS} and \cite{BD},  and the left-free subring is related to stable homotopy classes of selfmaps of the $p$-completion of the classifying space $BG$ of $G$, cf.~\cite{MP} and \cite{AKO}. For generalities on double Burnside rings and biset functors we refer the reader to Bouc's book \cite{BoucSLN}. 

\smallskip
For the bifree double Burnside ring we derive the following result. Let $\Sigmahat_G$ denote a set of representatives of the isomorphism classes of subgroups of $G$.

\begin{theorem}\label{thm bifree intro}
The primitive central idempotents $e_U$ of $B^\Delta(G,G)$ are parametrized by the elements $U\in\Sigmahat_G$ such that $U$ is a perfect group. The primitive central idempotents $e_{(U,\chi)}$ of $\QQ B^\Delta(G,G)$ are parametrized by a set $\calE_G$ of pairs $(U,\chi)$ with $U\in\Sigmahat_G$ and certain irreducible characters $\chi\in\Irr_{\QQ}(\Out(U))$, i.e., characters of irreducible modules of the outer automorphism group of $U$ over $\QQ$. More precisely, in $\QQ B^\Delta(G,G)$, one has $e_U=\sum_{\chi} e_{(V,\chi)}$, for each perfect $U\in\Sigmahat_G$, where $V$ is in $\Sigmahat_G$ such that $V^{(\infty)}\cong U$ and $\chi$ is in $\Irr_{\QQ}(\Out(V))$ such that $(V,\chi)\in\calE_G$.
\end{theorem}

Here, $V^{(\infty)}$ denotes the smallest normal subgroup of $V$ with solvable quotient. A precise definition of $\calE_G$ can be found in Remark~\ref{rem bifree idempotents over field}. Theorem~\ref{thm bifree intro} is a special case of Theorem~\ref{thm RBDelta}, which determines the primitive central idempotents of the ring $RB^\Delta(G,G)$, i.e., the scalar extension of $B^\Delta(G,G)$ from $\ZZ$ to $R$, for certain integral domains $R$, and of Remark~\ref{rem bifree idempotents over field}, in which more general fields $K$ are considered in place of $\QQ$.

\smallskip
For the left-free double Burnside ring we have the following result:

\begin{theorem}\label{thm left-free intro}
The center of $B^{\trl}(G,G)$ is connected, i.e., $0$ and $1$ are the only central idempotents of $B^{\trl}(G,G)$. The primitive central idempotents of $\QQ B^{\trl}(G,G)$ are contained in $\QQ B^\Delta(G,G)$. They are the sums $\sum_{(U,\chi)\in \calE} e_{(U,\chi)}$, where $\calE\subseteq\calE_G$ is an equivalence class of $\calE_G$ with respect to the transitive and symmetric closure of the relation defined on two elements $(U,\chi)$, $(U',\chi')\in\calE_G$ by
\begin{equation}\label{eqn relation intro}
  e_{(U,\chi)} \QQ B^{\trl}(G,G) e_{(U',\chi')}\neq \{0\}\,.
\end{equation}
\end{theorem}

The relation defined by (\ref{eqn relation intro}) is reformulated in explicit character-theoretic terms in Lemma~\ref{lem eBtrle'}. Theorem~\ref{thm left-free intro} is a special case of Theorem~\ref{thm left-free} and Corollary~\ref{cor general}, in which scalar extensions $RB^{\trl}(G,G)$ and $KB^{\trl}(G,G)$ for a larger class of integral domains $R$ (replacing $\ZZ$) and any field $K$ of characteristic $0$ (replacing $\QQ$) are considered.

\medskip
The paper is arranged as follows. Section~\ref{sec notation} introduces the notation used in the paper. In Section~\ref{sec bifree} we consider the bifree double Burnside ring $B^\Delta(G,G)$. The main theorem of this section, Theorem~\ref{thm RBDelta}, describes the primitive central idempotents of $RB^\Delta(G,G)$ for certain integral domains $R$. Remark~\ref{rem bifree idempotents over field} summarizes the case where $R$ is a field with some restrictions on the characteristic. Section~\ref{sec fusion systems} uses the same methods as Section~\ref{sec bifree} to show that the double Burnside ring $B^{\calF}(S,S)$ associated to a fusion system $\calF$ on a $p$-group $S$ has connected center, see Theorem~\ref{thm fusion}. The ring $B^{\calF}(S,S)$ was introduced in \cite{BD}; it is a subring of $B^\Delta(S,S)$. In Section~\ref{sec left-free} the primitive central idempotents of $RB^{\trl}(G,G)$ are studied for certain integral domains $R$. The main results are Theorem~\ref{thm left-free} and Corollary~\ref{cor general}, which imply Theorem~\ref{thm left-free intro}. Section~\ref{sec proof of lemma} is devoted solely to the technical proof of Lemma~\ref{lem eBtrle'}, which uses notation and a variety of results  from \cite{BD}. In Section~\ref{sec ex} we consider the example where $G$ is a cyclic group or an elementary abelian group.  


\section{Notation}\label{sec notation}

Throughout, $G$ denotes a finite group.

\begin{nothing} {\bf Generalities.}\quad 
The cardinality of a set $X$ is denoted by $|X|$.

\smallskip
We denote by $H\le G$ that $H$ is a subgroup of $G$ and by $H<G$ that $H$ is a proper subgroup of $G$. Similarly, $H\trianglelefteq G$ (resp. $H\vartriangleleft G$) denotes that $H$ is a normal (resp.~proper normal) subgroup of $G$. The trivial subgroup of $G$ will often be denoted by $1$. The group of automorphisms of $G$ is denoted by $\Aut(G)$. For an element $g$ of $G$, we denote by $c_g\in\Aut(G)$ the automorphism $x\mapsto gxg^{-1}$ of $G$. By $\Inn(G)$ we denote the group of inner automorphisms, namely $c_g$, $g\in G$, and by $\Out(G):=\Aut(G)/\Inn(G)$ we denote the outer automorphism group of $G$. For $H\le G$ and $g\in G$ we also write $\lexp{g}{H}$ instead of $gHg^{-1}$. If two subgroups $H$ and $K$ of $G$ are conjugate we write $H=_G K$ and if $H$ is conjugate to a subgroup of $K$ we write $H\le_G K$.

\smallskip
If $X$ is a left $G$-set and $x\in X$, we write $\stab_G(x)$ for the stabilizer of $x$ in $G$, and $X^G$ for the set of $G$-fixed points of $X$.
\end{nothing}

\begin{nothing} {\bf The (double) Burnside ring.}\quad
Recall that the Burnside ring $B(G)$ of $G$ is the Grothendieck group of the category of finite left $G$-sets with respect to the disjoint union of $G$-sets. The multiplication on $B(G)$ is induced by taking the direct product of $G$-sets. The element in $B(G)$ associated to a finite left $G$-set $X$ is denoted by $[X]$. If $H$ runs through a set of representatives of the conjugacy classes of subgroups of $G$, then the elements $[G/H]\in B(G)$, associated to the transitive $G$-sets $G/H$, form a $\ZZ$-basis of $B(G)$. For any subgroup $H$ of $G$ we have a ring homomorphism $\Phi_H\colon B(G)\to \ZZ$ determined by $\Phi_H([X])=|X^H|$, for any finite left $G$-set $X$. We refer the reader to \cite[\S80A]{CR} or \cite[Chapter~2]{BoucSLN} for basic facts on the Burnside ring.

\smallskip
For two finite groups $G$ and $H$, the double Burnside group $B(G,H)$ is the Grothendieck group of the category of finite $(G,H)$-bisets $X$, i.e., finite sets with a left $G$-action and a right $H$-action that commute with each other, with respect to disjoint unions. As a special case, we obtain the double Burnside ring $B(G,G)$ whose multiplication is induced by taking the tensor product $X\times_G Y$ of two $(G,G)$-bisets $X$ and $Y$. This is the set of $G$-orbits $x\times_G y$ of elements $(x,y)\in X\times Y$ under the $G$-action $g(x,y):=(xg^{-1},gy)$. We often identify $(G,H)$-biset structures on a set $X$ with left $G\times H$-set structures on the same set $X$ via $(g,h)x=gxh^{-1}$ for $x\in X$, $g\in G$ and $h\in H$. With this identification we can identify $B(G,H)$ and $B(G\times H)$ as additive groups. Note that the abelian group $B(G\times G)$ now has two ring structures, the first one given by the direct product construction, the second one by the tensor product construction on $B(G,G)$. We denote the first one just by \lq\lq$\cdot$\rq\rq\ and the second one by \lq\lq$\cdot_G$\rq\rq. For more details we refer the reader to \cite[Chapter~2]{BoucSLN}.

\smallskip
If $G$ and $H$ are finite groups and if $L\le G\times H$ is a subgroup, we denote by $p_1\colon G\times H\to G$ and $p_2\colon G\times H\to H$ the two projection maps and we set $k_1(L):=\{g\in G\mid (g,1)\in L\}$ and $k_2(L):=\{h\in H\mid (1,h)\in L\}$. 
Then $k_i(L)\trianglelefteq p_i(L)$ for $i=1,2$, and $\eta(L)\colon p_2(L)/k_2(L)\to p_1(L)/k_1(L)$, defined by $hk_2(L)\mapsto gk_1(L)$ whenever $(g,h)\in L$, is a well-defined isomorphism. 
This way one obtains a bijection between the set of subgroups $L$ of $G\times H$ and the quintuples $(P_1,K_1,\eta,P_2,K_2)$ with $K_1\trianglelefteq P_1\le G$, $K_2\trianglelefteq P_2\le H$, and $\eta\colon P_2/K_2\myiso P_1/K_1$, cf.~\cite[Lemma~2.3.25]{BoucSLN}. With this notation, a $(G,G)$-biset $X$ is left-free (resp.~bifree) if and only if each stabilizer $L$ of an element of $X$ satisfies $k_1(L)=1$ (resp.~$k_1(L)=1=k_2(L)$). Thus,  the corresponding left-free double Burnside ring $B^{\trl}(G,G)$ (resp.~bifree double Burnside ring $B^\Delta(G,G)$) is a free $\ZZ$-module with basis elements $[G\times G/L]$, where $L$ runs through a set of representatives of $G\times G$-conjugacy classes of subgroups of $G\times G$ with $k_1(L)=1$ (resp.~$k_1(L)=1=k_2(L)$). For an isomorphism $\phi\colon V\myiso U$ between subgroups of $G$ we set $\Delta(U,\phi,V):=\{(\phi(v),v)\mid v\in V\}$, the subgroup corresponding to $(U,1,\phi,V,1)$. If $U=V$ and $\phi=\id_U$ we also write $\Delta(U)$.
\end{nothing}


\section{Central idempotents of $RB^\Delta(G,G)$}\label{sec bifree}

Throughout this section $G$ denotes a finite group and $R$ denotes a commutative ring. We denote by $\Sigma_G$ the set of subgroups of $G$, by $\Sigmatilde_G\subseteq \Sigma_G$ a set of representatives of the conjugacy classes of $\Sigma_G$, and by $\Sigmahat_G\subseteq \Sigmatilde_G$ a set of representatives of the isomorphism classes of $\Sigma_G$.

\begin{lemma}\label{lem trans perm module}
Let  $R$ be an integral domain and let $X$ be a transitive $G$-set. If no prime divisor of $|X|$ is invertible in $R$ then the $RG$-permutation module $RX$ is indecomposable.
\end{lemma}

\begin{proof}
We may assume that $|X|\neq 1$. Assume that $RX = M\oplus N$ is a direct sum decomposition into $RG$-submodules $M$ and $N$, and assume that $M\neq\{0\}$. Then $M$ and $N$ are finitely generated projective $R$-modules and they have a well-defined $R$-rank, cf.~\cite[Theorem~I.4.12]{DI}. Let $x\in X$, set $H:=\stab_G(x)$, and let $p$ be a prime divisor of $|X|=[G:H]$. Since $pR\neq R$, there exists a maximal ideal $P$ of $R$ such that $p\in P$. Then $F:=R/P$ is a field of characteristic $p$. Let $\Fbar$ denote an algebraic closure of $F$. Then $\Fbar X\cong \Fbar\otimes_R M \oplus \Fbar\otimes_R N$, where $\Fbar X$, $\Fbar\otimes_R M$ and $\Fbar\otimes _R N$ are relatively $H$-projective $\Fbar G$-modules. By a result of Green, the $p$-part $[G:H]_p=|X|_p$ of $|X|$ divides
\begin{equation*}
  \dim_{\Fbar} (\Fbar\otimes_R M) =\dim_F (F\otimes_R M) = \rk_{R_P}(R_P\otimes_R M) = \rk_R (M)\,.
\end{equation*}
Since $p$ is arbitrary, we conclude that $|X|$ divides $\rk_R(M)$. But
\begin{equation*}
  0\neq \rk_R(M)\le \rk_R(M)+\rk_R (N) = \rk_R (M\oplus N) = \rk_R(RX) = |X|,
\end{equation*}
which implies $\rk_R(M)=|X|$ and $\rk_R(N)=0$. Thus, $N=0$ and $M=RX$.
\end{proof}

\begin{remark}\label{rem central idempotents}
Let $\Lambda$ be a ring and let $1_\Lambda= e_1+\cdots+e_n$ be a decomposition of $1_\Lambda$ into primitive pairwise orthogonal idempotents $e_1,\ldots,e_n$ of $\Lambda$. Then every central idempotent $e$ of $\Lambda$ is equal to the subsum $e=\sum_{i\in I} e_i$, where $I$ denotes the set of all elements $i\in \{1,\ldots,n\}$ satisfying $e_ie=e_i$. In fact, let $i\in\{1,\ldots,n\}$ be arbitrary. Then $ee_i$ and $(1_\Lambda-e)e_i$ are orthogonal idempotents with $ee_i+(1-e)e_i=e_i$. Since $e_i$ is primitive, we obtain $ee_i=e_i$ or $ee_i=0$. By multiplying the equation $1_\Lambda=e_1+\cdots+e_n$ on both sides with $e$, we now obtain the desired expression for $e$.
\end{remark}

\begin{proposition}\label{prop idempotents of Hecke algebras}
Let $X$ be a finite $G$-set and let $R$ be an integral domain. Assume that, for every $x\in X$ and for every prime divisor $p$ of $[G:\stab_G(x)]$, one has $\{0\}\neq pR\neq R$. Then the ring $\End_{RG}(RX)$ has no central idempotent different from $0$ and $1$.
\end{proposition}

\begin{proof}
Let $K$ denote the field of fractions of $R$. We decompose $X$ into $G$-orbits, $X=X_1\coprod\cdots\coprod X_n$, and obtain decompositions
\begin{equation}\label{eqn X-decomp}
  RX = RX_1\oplus\cdots\oplus RX_n\quad\text{and}\quad KX = KX_1 \oplus \cdots \oplus KX_n\,
\end{equation}
into $RG$-submodules and $KG$-submodules, respectively. We decompose $KX_i$, for each $i=1,\ldots,n$, into indecomposable $KG$-submodules:
\begin{equation}\label{eqn X_i-decomposition}
  KX_i = V_i^{(1)} \oplus \cdots \oplus V_i^{(r_i)}\,.
\end{equation}
We may assume that $V_i^{(1)}\cong K$, the trivial $KG$-module. In fact, the hypothesis on $R$ and $X$ implies that $|X_i|\neq 0$ in $K$. This implies that $\iota\colon K\to KX_i$, $1\mapsto |X_i|^{-1}\sum_{x\in X_i}x$, and $\pi\colon KX_i\to K$, $x\mapsto 1$, are $KG$-module homomorphisms with $\pi\circ\iota=\id_K$, so that $K$ is isomorphic to a direct summand of $KX_i$. Let $e_i\in\End_{RG}(RX)$ denote the idempotent which is the projection onto the $i$-th component in the first decomposition in (\ref{eqn X-decomp}). Then $e_i$ is primitive in $\End_{RG}(RX)$, by Lemma~\ref{lem trans perm module}. We view $\End_{RG}(RX)$ as a subring of $\End_{KG}(KX)$ via the canonical embedding and decompose $e_i$ in $\End_{KG}(KX)$ further into primitive idempotents corresponding to the decomposition in (\ref{eqn X_i-decomposition}):
\begin{equation*}
  e_i=e_i^{(1)}+\cdots+e_i^{(r_i)}\,.
\end{equation*}
Altogether we have a primitive decomposition
\begin{equation}\label{eqn prim decomp}
  1 = (e_1^{(1)} + e_1^{(2)} + \cdots + e_1^{(r_1)}) + \cdots + (e_n^{(1)} + \cdots + e_n^{(r_n)})
\end{equation}
in $\End_{KG}(KX)$. Now let $e$ be a non-zero central idempotent of $\End_{RG}(RX)$. Since $1=e_1+\cdots + e_n$ is a primitive decomposition of $1$ in $\End_{RG}(RX)$, we have $e=\sum_{i\in I} e_i$ for some $\emptyset\neq I\subseteq \{1,\ldots,n\}$, by Remark~\ref{rem central idempotents}. Since $e$ is also a central idempotent of $\End_{KG}(KX)$, it is also a subsum of the decomposition in (\ref{eqn prim decomp}). Since $I\neq \emptyset$, there exists an element $i\in I$, and we have $e_ie= e_i$. This implies that $e_i^{(1)} e = e_i^{(1)}$. For every $j\in\{1,\ldots,n\}$ there exists an isomorphism $\alpha\colon KX\to KX$ such that $\alpha e_i^{(1)}\alpha^{-1} = e_j^{(1)}$. The equation $e_i^{(1)}e = e_i^{(1)}$ implies
\begin{equation*}
  e_j^{(1)} = \alpha e_i^{(1)}\alpha^{-1} = \alpha e_i^{(1)} e \alpha^{-1} = \alpha e_i^{(1)}\alpha^{-1} e 
  = e_j^{(1)} e\,.
\end{equation*}
This implies that $e_j e\neq 0$, and Remark~\ref{rem central idempotents} implies that $j\in I$. Thus $I=\{1,\ldots,n\}$ and $e=1$.
\end{proof}

Recall that we have a ring homomorphism $\Delta\colon B(G)\to B^{\Delta}(G,G)$, given by $\Delta([G/U])=[G\times G/\Delta(U)]$. The following lemma is proved in \cite{BP}.

\begin{lemma}\label{lem BP lemma} 
Let $a\in B(G)$ and let $U\le G$. Then
\begin{equation*}
  \Phi_{\Delta(U)}(\Delta(a)) = |C_G(U)|\cdot\Phi_U(a)\,.
\end{equation*}
\end{lemma}

\begin{remark}\label{rem sigma} 
Recall from \cite[5.3--5.5]{BD} that the map
\begin{align*}
  \sigma_G\colon B^\Delta(G,G) & \to \prod_{U\in\Sigmahat_G}\End_{\ZZ \Out(U)}(\ZZ\Injbar(U,G))\,,\\
  a & \mapsto \Bigl([\mu]\mapsto\sum_{[\lambda]\in\Injbar(U,G)} 
     \frac{\Phi_{\Delta(\lambda(U),\lambda\mu^{-1},\mu(U))}(a)}{|C_G(\lambda(U))|}
        \cdot[\lambda]\Bigr)_{U\in\Sigmahat_G}
\end{align*}
is a well-defined injective ring homomorphism with image of finite index which induces an $R$-algebra homomorphism
\begin{equation*}
  \sigma_G\colon RB^\Delta(G,G) \to \prod_{U\in\Sigmahat_G}\End_{R\Out(U)}(R\Injbar(U,G))
\end{equation*}
for every commutative ring $R$. The latter homomorphism is an isomorphism if $|G|$ is invertible in $R$. In particular, if $a\in Z(B^\Delta(G,G))$ then $\sigma_G(a)$ is central in $\prod_{U\in\Sigmahat_G}\End_{\ZZ \Out(U)}(\ZZ\Injbar(U,G))$ and in $ \prod_{U\in\Sigmahat_G}\End_{R\Out(U)}(R\Injbar(U,G))$. Here, for $U\in\Sigmahat_G$, $\Inj(U,G)$ denotes the $(G,\Aut(U))$-biset of injective group homomorphisms from $U$ to $G$ with $g\cdot\lambda\cdot\omega:=c_g\circ\lambda\circ\omega$ for $g\in G$, $\lambda\in\Inj(U,G)$, and $\omega\in\Aut(U)$. Finally, $\Injbar(U,G):=G\backslash\Inj(U,G)$ is the set of $G$-orbits with the induced right action of $\Out(U)$. The $G$-orbit of $\lambda\in\Inj(U,G)$ is denoted by $[\lambda]$. We fix a subgroup $U\in\Sigmahat_G$. Let $U_1,\ldots, U_r\in\Sigmatilde_G$ be the representatives of the $G$-conjugacy classes of all subgroups of $G$ which are isomorphic to $U$. Then the right $\Out(U)$-set $\Injbar(U,G)$ decomposes into orbits,
\begin{equation*}
  \Injbar(U,G)=\coprod_{i=1}^r \Injbar(U,G)_i\,,
\end{equation*}
where $\Inj(U,G)_i$ denotes the set of elements $\lambda\in\Inj(U,G)$ such that $\lambda(U)$ is $G$-conjugate to $U_i$, and $\Injbar(U,G)_i$ denotes the set of $G$-orbits formed by such elements.
In particular, we obtain a decomposition into $R\Out(U)$-submodules:
\begin{equation}\label{eqn RInjbar decomp}
  R\Injbar(U,G)=\bigoplus_{i=1}^r R\Injbar(U,G)_i\,.
\end{equation}
Finally, note that, for every $i=1,\ldots,r$, the map $\lambda\mapsto [\lambda]$ induces a bijection
\begin{equation*}
  N_G(U_i)\backslash\Inj(U,U_i)\liso \Injbar(U,G)_i
\end{equation*}
of $\Out(U)$-sets, and that $N_G(U_i)\backslash \Inj(U,U_i)\cong A_i\backslash\Aut(U)$ as $\Out(U)$-sets, where $A_i:=\lambda^{-1}B_i\lambda\le\Aut(U)$ denotes the subgroup corresponding to the image $B_i$ of the map $N_G(U_i)\to\Aut(U_i)$, $g\mapsto c_g$, under any isomorphism $\lambda\colon U\liso U_i$.
\end{remark}

\begin{remark}\label{rem idempotents in B(G)}
In this remark we assume that $R$ is an integral domain with field of fractions $K$ such that $|G|$ is invertible in $K$. We denote by $\pi$ the set of prime divisors of $|G|$ which are not invertible in $R$. By $\Theta^\pi_G$ we denote the set of $\pi$-perfect subgroups of $G$, i.e., the subgroups $U$ of $G$ with the property that $U$ has no factor group of prime order $p\in\pi$. For any group $G$ we denote by $G^{(\pi)}$ the smallest normal subgroup of $G$ with solvable factor group of $\pi$-order, i.e., of order only divisible by primes from $\pi$. Clearly, $G^{(\pi)}$ is $\pi$-perfect. Thus, $U^{(\pi)}\in\Theta_G^\pi$ for every $U\le G$. We further define $\Thetatilde_G^{\pi}:=\Theta_G^{\pi}\cap\Sigmatilde_G$ and $\Thetahat_G^{\pi}:=\Theta_G^{\pi}\cap\Sigmahat_G$. Thus, $\Thetahat_G^{\pi}\subseteq\Thetatilde_G^{\pi}$ are sets of representatives of the isomorphism classes and of the conjugacy classes of $\pi$-perfect subgroups of $G$, respectively.

A variation of the arguments in \cite[Corollary~3.3.6]{BcHB} gives the following description of the primitive idempotents of $RB(G)$ and of $KB(G)$. 
For each $U\in\Sigma_G$, let $e_U\in KB(G)$ denote the unique element with the property that $\Phi_{U'}(e_U)=1$ if $U=_G U'$, and $\Phi_{U'}(e_U)=0$ if $U'\neq_G U$. Then the elements $e_U$, $U\in\Sigmatilde_G$, form a set of primitive, pairwise orthogonal idempotents of $KB(G)$ whose sum is equal to $1$. 

For every $U\in\Theta_G^{\pi}$ set
\begin{equation*}
  \varepsilon_{U}^{(\pi)}:=\sum_{\substack{V\in\Sigmatilde_{G}\\ V^{(\pi)}=_G U}} e_V\,.
\end{equation*}
Then the elements $\varepsilon_{U}^{(\pi)}$, $U\in\Thetatilde_G^\pi$, are primitive, pairwise orthogonal idempotents of $RB(G)$ whose sum is equal to $1$. 
\end{remark}

\begin{proposition}\label{prop image of e_U}
Assume that $K$ is a field such that $|G|$ is invertible in $K$. Further
assume the notation established in Remarks~\ref{rem sigma} and \ref{rem idempotents in B(G)}. Let $V\le G$, and let $(f_U)_{U\in\Sigmahat_G}$ denote the image of $e_V\in KB(G)$ under the $K$-algebra homomorphism
\begin{equation*}
  KB(G) \Ar{\Delta} KB^\Delta(G,G)\Ar{\sigma_G} \prod_{U\in\Sigmahat_G} \End_{K\Out(U)}(K\Injbar(U,G))\,.
\end{equation*}
If $U\in\Sigmahat_G$ is not isomorphic to $V$ then $f_U=0$. If $U\in\Sigmahat_G$ is isomorphic to $V$ then $f_U$ is equal to the projection onto the direct summand of $K\Injbar(U,G)$ with respect to the decomposition in (\ref{eqn RInjbar decomp}), with $K$ in place of $R$, which is indexed by the $G$-conjugacy class containing $V$.
\end{proposition}

\begin{proof}
Let $U\in\Sigmahat_G$ and let $\mu\in\Inj(U,G)$. Then
\begin{equation*}
  f_U([\mu])= \sum_{[\lambda]\in\Injbar(U,G)} 
  \frac{\Phi_{\Delta(\lambda(U),\lambda\mu^{-1},\mu(U))}(\Delta(e_V))}{|C_G(\lambda(U))|} [\lambda]\,.
\end{equation*}
Suppose that $f_U([\mu])\neq0$. We will show that then $\mu(U)$ is $G$-conjugate to $V$ and that $f_U([\mu])=[\mu]$.

Since $e_V$ is a linear combination of elements of the form $[G/W]$ with $W\le_G V$, the idempotent $\Delta(e_V)$ is a linear combination of elements of the form $[G\times G/\Delta(W)]$ with $W\le_G V$. Thus, $\Phi_{\Delta(\lambda(U),\lambda\mu^{-1},\mu(U))}([G\times G/\Delta(W)]) \neq 0$ for some $W\le_G V$, and therefore $\Delta(\lambda(U),\lambda\mu^{-1},\mu(U))\le_{G\times G} \Delta(W)$. Hence, $\Delta(\lambda(U),\lambda\mu^{-1},\mu(U))=_{G\times G}\Delta(X)$ for some $X\le W\le_G V$, and 
\begin{equation*}
  0\neq \Phi_{\Delta(\lambda(U),\lambda\mu^{-1},\mu(U))} (\Delta(e_V)) = 
  \Phi_{\Delta(X)}(\Delta(e_V)) = |C_G(X)|\Phi_X(e_V)\,,
\end{equation*}
by Lemma~\ref{lem BP lemma}. Thus, $X=_G V$ and $\Delta(\lambda(U),\lambda\mu^{-1},\mu(U)) =_{G\times G}\Delta(V)$. Let $g,h\in G$ be such that
\begin{equation*}
  \Delta(V) =\lexp{(g,h)}{\Delta(\lambda(U),\lambda\mu^{-1},\mu(U))} = 
  \{(g\lambda(u)g^{-1}, h\mu(u)h^{-1})\mid u\in U\}\,.
\end{equation*}
Since $[\lambda]=[c_g\lambda]$ and $[\mu]=[c_h\mu]$, we may assume that $\Delta(V)=\Delta(\lambda(U),\lambda\mu^{-1},\mu(U))=\{(\lambda(u),\mu(u))\mid u\in U\}$. Thus, $\lambda=\mu$, $\lambda(U)=V$, and 
\begin{equation*}
  f_U([\mu]) = \frac{\Phi_{\Delta(V)}(\Delta(e_V))}{|C_G(V)|} [\mu]=\Phi_V(e_V)[\mu] = [\mu]\,,
\end{equation*}
again by Lemma~\ref{lem BP lemma}.
\end{proof}


\begin{theorem}\label{thm RBDelta}
Let $R$ be an integral domain with field of fractions $K$ such that $|G|$ is invertible in $K$. Assume further that for every $U\in\Sigmahat_G$ and every prime divisor $p$ of $|\Out(U)|$ one has $\{0\}\neq pR \neq R$.
Let $\pi$ denote the set of prime divisors of $|G|$ which are not invertible in $R$. Assume the notation from Remark~\ref{rem idempotents in B(G)}. The primitive central idempotents of $RB^\Delta(G,G)$ are parametrized by isomorphism classes of $\pi$-perfect subgroups of $G$. More precisely, for $W\in\Thetahat_G^{\pi}$, set
\begin{equation*}
 \varepsilonhat_W^{(\pi)}:=\sum_{W\cong V\in\Thetatilde_G}\varepsilon_V^{(\pi)} \in RB(G)\,.
\end{equation*}
Then the elements $\Delta(\varepsilonhat_W^{(\pi)})$, $W\in\Thetahat_G^{\pi}$, are primitive, pairwise orthogonal idempotents of $Z(RB^\Delta(G,G))$ whose sum is equal to $1$.
\end{theorem}

\begin{proof}
We will make use of the commutative diagram
\begin{diagram}
  \movevertex(-50,0){RB(G)} & \movearrow(-50,0){\Ear{\Delta}}  & \movevertex(-50,0){RB^{\Delta}(G,G)} &
        \movearrow(-45,0){\Ear[30]{\sigma_G}} & 
        \movevertex(0,-10){\mathop{\prod}\limits_{U\in\Sigmahat_G} \End_{R\Out(U)}(R\Injbar(U,G))} &&
  \movearrow(-50,0){\sar} &  & \movearrow(-50,0){\sar} &  & \movearrow(0,0){\sar} &&
  \movevertex(-50,0){KB(G)} & \movearrow(-50,0){\Ear{\Delta}} & \movevertex(-50,0){KB^\Delta(G,G)} &
          \movearrow(-45,0){\Ear[30]{\sigma_G}} &
          \movevertex(0,-10){\mathop{\prod}\limits_{U\in\Sigmahat_G} \End_{K\Out(U)}(K\Injbar(U,G))} &&
\end{diagram}
whose vertical maps are the canonical embeddings. All maps in the diagram are injective and the map $\sigma_G$ of the bottom row is an isomorphism.

\smallskip
(a) First we show that each element $\Delta(\varepsilonhat_W^{(\pi)})$, $W\in\Thetahat_G^{\pi}$, is a central idempotent of $RB^\Delta(G,G)$. For $W\in\Thetahat_G^{\pi}$, the element $\varepsilonhat_W^{(\pi)}:= \sum_{W\cong V \in\Thetatilde_G} \varepsilon_V^{(\pi)}$ is an idempotent of $RB(G)$, by Remark~\ref{rem idempotents in B(G)}. Therefore, $\Delta(\varepsilonhat_W^{(\pi)})$ is an idempotent of $RB^\Delta(G,G)$. To see that it is central in $RB^\Delta(G,G)$ it suffices to show that $\sigma_G(\Delta(\varepsilonhat_W^{(\pi)}))$ is central in $\prod_{U\in\Sigmahat_G}\End_{K\Out(U)}(K\Injbar(U,G))$. But, by Proposition~\ref{prop image of e_U}, the $U$-component of $\sigma_G(\Delta(\varepsilonhat_W^{(\pi)}))$ is equal to the identity map if $U^{(\pi)}\cong W$ and it is equal to $0$ if $U^{(\pi)}\not\cong W$. So clearly, this element is central. 

\smallskip
(b) Next we show that each element $\Delta(\varepsilonhat_W^{(\pi)})$, $W\in\Thetahat_G^{\pi}$, is primitive in $Z(RB^\Delta(G,G))$. Let $W\in\Thetahat_G^{\pi}$ and let $e$ be a primitive central idempotent of $RB^\Delta(G,G)$ with $e\cdotG\Delta(\varepsilonhat_W^{(\pi)})=e$. Then $\sigma_G(e)$ is a central idempotent of $\prod_{U\in\Sigmahat_G}\End_{R\Out(U)} (R\Injbar(U,G))$. By Proposition~\ref{prop idempotents of Hecke algebras} there exists a subset $\Xi\subset\Sigma_G$ which is closed under taking isomorphic subgroups such that, with $\Xihat:=\Xi\cap\Sigmahat_G$, the $U$-component of $\sigma_G(e)$ is equal to the identity if $U\in\Xihat_G$ and equal to $0$ if $U\notin\Xihat_G$. Now Proposition~\ref{prop image of e_U} implies that 
\begin{equation*}
  \sigma_G(e)=\sigma_G(\Delta(\sum_{U\in\Xitilde}e_U))\,,
\end{equation*}
where $\Xitilde=\Xi\cap\Sigmatilde_G$. This implies that
\begin{equation*}
  e=\Delta(\sum_{U\in\Xitilde}e_U)\in RB^\Delta(G,G)\cap \Delta(KB(G)) = \Delta(RB(G))\,.
\end{equation*}
The injectivity of $\Delta$ implies that $\sum_{U\in\Xitilde} e_U\in RB(G)$. Since $e\neq 0$ and since $e\cdotG\Delta(\varepsilonhat_W^{(\pi)})=e$, we know that $\Xi$ contains a subgroup $U$ of $G$ satisfying $U^{(\pi)}\cong W$. Moreover, since $\sum_{U\in\Xitilde}e_U$ is an idempotent in $RB(G)$, we obtain that $\Xi$ contains a subgroup which is isomorphic to $W$. Since $\Xi$ is closed under taking isomorphic subgroups, $\Xi$ contains all subgroups of $G$ which are isomorphic to $W$. Again, since $\sum_{U\in\Xitilde}e_U$ is an element of $RB(G)$, $\Xi$ contains all subgroups $U$ of $G$ with $U^{(\pi)}\cong W$. This implies that $e\cdotG\Delta(\varepsilonhat_W^{(\pi)})=\Delta(\varepsilonhat_W^{(\pi)})$ and therefore, $e=e\cdotG\Delta(\varepsilonhat_W^{(\pi)})=\Delta(\varepsilonhat_W^{(\pi)})$. Thus, $\Delta(\varepsilonhat_W^{(\pi)})$ is primitive in $Z(RB^\Delta(G,G))$. 

\smallskip
(c) Finally, 
\begin{equation*}
  \sum_{W\in\Thetahat_G^{\pi}} \Delta(\varepsilonhat_W^{(\pi)}) = 
  \Delta\bigl(\sum_{W\cong V\in\Thetatilde_G^{\pi}}\varepsilon_V^{(\pi)}\bigr) =
  \Delta(\sum_{U\in\Sigmatilde_G}e_U) = \Delta(1) = 1\,,
\end{equation*}
and the proof is complete.
\end{proof}

In the following remark we will determine the primitive central idempotents of $KB^\Delta(G,G)$ for certain fields $K$. This will be used in Section~\ref{sec left-free} in the more restricted situation that $K$ has characteristic $0$. Euler's totient function will be denoted by $\varphi$.

\begin{remark}\label{rem bifree idempotents over field}
Let $K$ be a field such that $|G|$, $|\Out(U)|$ and $\varphi(|\Out(U)|)$, for $U\le G$, are invertible in $K$. 

\smallskip
(a) Recall from Remark~\ref{rem sigma} that the map
\begin{equation}\label{eqn sigma over K}
  \sigma_G\colon KB^\Delta(G,G)\liso \prod_{U\in\Sigmahat_G}\End_{K\Out(U)}(K\Injbar(U,G))
\end{equation}
is an isomorphism of $K$-algebras. Moreover, for each $U\in\Sigmahat_G$, the $K\Out(U)$-module $K\Injbar(U,G)$ is semisimple. Let $\calE_G$ denote the set of pairs $(U,\chi)$ with $U\in\Sigmahat_G$ and $\chi\in\Irr_K(\Out(U))$ such that $\chi$ occurs as a constituent in the character of $K\Injbar(U,G)$. Then, by the above isomorphism, the primitive central idempotents of $KB^\Delta(G,G)$ are given by the elements $e_{(U,\chi)}$, $(U,\chi)\in\calE_G$, where, for $U'\in\Sigmahat_G$, the $U'$-component of $\sigma_G(e_{(U,\chi)})$ is equal to $0$ if $U'\neq U$ and equal to the map $a\mapsto a\cdot e_\chi$, for $a\in K\Injbar(U,G)$. Here, $e_\chi$ denotes the primitive central idempotent of $K\Out(U)$ associated to the irreducible character $\chi$. Note that one has
\begin{equation}\label{eqn e_chi}
  e_\chi=\frac{\chi(1)}{s^2 r |\Out(U)|} \sum_{\omegabar\in\Out(U)} \chi(\omegabar^{-1})\,\omegabar\in K\Out(U)\,,
\end{equation}
if $\chi=s(\psi_1+\cdots+\psi_r)$ is a decomposition of $\chi$ into absolutely irreducible characters over some extension field of $K$. 
Note that $s$ is invertible in $K$, since $s$ divides $|\Out(U)|$ in the case that $\charac(K)=0$ and since $s=1$ if $\charac(K)\neq 0$. 
(Recall that the Schur index is one in positive characteristic.)
Also note that $r$ is invertible in $K$, since $r$ is the degree of a subextension of the extension $K(\zeta)/K$, where $\zeta$ is a root of unity of order $|\Out(U)|$. 

\smallskip
(b) For $U\in\Sigmahat_G$, let $\Sigmatilde_G(U)$ denote the set of elements $V\in\Sigmatilde_G$ with $V\cong U$. We can rewrite the decomposition (\ref{eqn RInjbar decomp}) in Remark~\ref{rem sigma} as indexed over $\Sigmatilde_G(U)$:
\begin{equation}\label{eqn KInjbar decomp}
  K\Injbar(U,G) = \bigoplus_{V\in\Sigmatilde_G(U)} K\Injbar(U,G)_V\,.
\end{equation}
Then, for $(U,\chi)\in\calE_G$ and each $V\in\Sigmatilde_G(U)$, the element $e_{(U,\chi,V)}$, defined by requiring that, for $U'\in\Sigmahat_G$, the $U'$-component of $\sigma_G(e_{(U,\chi,V)})$ is equal to $0$ if $U'\neq U$, and that the $U$-component is equal to $0$ in all components of the decomposition (\ref{eqn KInjbar decomp}) different from $V$, and finally equal to \lq\lq multiplication with $e_\chi$\rq\rq\ in the $V$-component, is an idempotent of $KB^\Delta(G,G)$. This leads to a decomposition
\begin{equation}\label{eqn eUchiV decomp}
  e_{(U,\chi)} = \sum_{V\in\Sigmatilde_G(U)} e_{(U,\chi,V)}
\end{equation}
of the primitive idempotent of $Z(KB^\Delta(G,G))$ as a sum of pairwise orthogonal idempotents in $KB^\Delta(G,G)$.
\end{remark}


\section{Central idempotents of $RB^{\calF}(S,S)$ for a fusion system $\calF$ on a $p$-group $S$}\label{sec fusion systems}

Throughout this section we fix a $p$-group $S$ and a (not necessarily saturated) fusion system $\calF$ on $S$. For definitions and basic results on fusion systems we refer the reader to \cite{AKO}. In \cite{BD}, a subring $B^\calF(S,S)$ of $B^\Delta(S,S)$ was constructed which is defined as the $\ZZ$-span of the standard basis elements $[S\times S/\Delta(P,\phi,Q)]$, where $\phi\colon Q\liso P$ runs through all isomorphisms in the category $\calF$. We call $B^\calF(S,S)$ the {\em double Burnside ring of $\calF$}. In this section we will show that $B^{\calF}(S,S)$ has no central idempotent different from $0$ and $1$. 

\begin{remark}\label{rem sigmaF}
In this remark we recall some notation and some results from \cite{BD}. Again we denote by $\Sigma_S$ the set of subgroups of $S$, by $\Sigmatilde_S\subseteq\Sigma_S$ a set of representatives of the $S$-conjugacy classes of subgroups of $S$, and by $\Sigmahat_S^{\calF}\subseteq\Sigmatilde_S$ a set of representatives of the $\calF$-isomorphism classes of subgroups of $S$. It was shown in \cite[Theorem~5.7]{BD} (see also \cite[Subsection 7.11]{BD}) that the map
\begin{align*}
  \sigmatilde_S^{\calF}\colon B^{\calF}(S,S) 
  & \to \prod_{P\in\Sigmahat_S^{\calF}} \End_{\ZZ\Out_{\calF}(P)} (\ZZ \Hombar_{\calF}(P,S))\,, \\
  a & \mapsto \bigl([\psi]\mapsto\sum_{[\phi]\in\Hombar_{\calF}(P,S)} 
     \frac{\Phi_{\Delta(\phi(P),\phi\psi^{-1},\psi(P))}(a)}{|C_S(P)|} \cdot [\phi]\bigr)_P\,,
\end{align*}
is a well-defined injective ring homomorphism with finite cokernel which induces an $R$-algebra homomorphism
\begin{equation*}
 \sigmatilde_S^{\calF}\colon R B^{\calF}(S,S) 
     \to \prod_{P\in\Sigmahat_S^\calF} \End_{R\Out_{\calF}(P)} (R \Hombar_{\calF}(P,S))
\end{equation*}
for every commutative ring $R$. If $p$ is invertible in $R$ then the latter homomorphism is an isomorphism. Here, $\Hombar_{\calF}(P,S)$ denotes the set of $S$-orbits of $\Hom_{\calF}(P,S)$ under the action $x\cdot\phi:=c_x\circ \phi$ for $x\in S$ and $\phi\in\Hom_{\calF}(P,S)$. The set $\Hombar_{\calF}(P,S)$ has a right action of the group $\Out_{\calF}(P):=\Aut_{\calF}(P)/\Inn(P)$, which is given by composition. The $S$-orbit of $\phi\in\Hom_{\calF}(P,S)$ is denoted by $[\phi]$.

We fix a subgroup $P\in\Sigmahat_S^{\calF}$. Assume that $P_1,\ldots,P_r\in\Sigmatilde_S$ are representatives of the conjugacy classes of subgroups of $S$ which are $\calF$-isomorphic to $P$. Then the right $\Out_{\calF}(P)$-set $\Hombar_{\calF}(P,S)$ decomposes into orbits
\begin{equation*}
  \Hombar_{\calF}(P,S) = \coprod_{i=1}^r \Hombar_{\calF}(P,S)_i\,,
\end{equation*}
where $\Hom_{\calF}(P,S)_i$ denotes the set of elements $\phi\in\Hom_{\calF}(P,S)$ such that $\phi(P)$ is $S$-conjugate to $P_i$, and $\Hombar_{\calF}(P,S)_i$ denotes the set of $S$-orbits of $\Hom_{\calF}(P,S)_i$. In particular, we obtain a decomposition
\begin{equation}\label{eqn RHombar decomp}
  R\Hombar_{\calF}(P,S) = \bigoplus_{i=1}^r R\Hombar_{\calF}(P,S)_i
\end{equation}
into $R\Out_{\calF}(P)$-submodules.
\end{remark}

The following proposition can be proved in a completely analogous way as Proposition~\ref{prop image of e_U}.

\begin{proposition}\label{prop image of e_U for F}
Let $K$ be a field of characteristic different from $p$ and assume the notation from Remarks~\ref{rem sigmaF} and \ref{rem idempotents in B(G)}. Furthermore, let $Q\le S$ and let $(f_P)_{P\in\Sigmahat_S^{\calF}}$ denote the image of $e_Q\in KB(S)$ under the $K$-algebra homomorphism
\begin{equation*}
  KB(S) \Ar{\Delta} KB^{\calF}(S,S) \Ar{\sigmatilde_S^{\calF}} \prod_{P\in\Sigmahat_S^\calF} 
  \End_{K\Out_{\calF}(P)} (K\Hombar_{\calF}(P,S))\,.
\end{equation*}
If $P\in\Sigmahat_S^{\calF}$ is not $\calF$-isomorphic to $Q$ then $f_P=0$. If $P\in\Sigmahat_S^{\calF}$ is $\calF$-isomorphic to $Q$ then the endomorphism $f_P$ is equal to the projection onto the direct summand of $K\Hombar_{\calF}(P,S)$  with respect to the decomposition in (\ref{eqn RHombar decomp}), with $K$ in place of $R$, which is indexed by the $S$-conjugacy class containing $Q$.
\end{proposition}

For every $P\le S$ we denote by $\Aut_S(P)\le \Aut(P)$ the image of the map $N_S(P)\to\Aut(P)$, $g\mapsto c_g$.

\begin{theorem}\label{thm fusion}
Let $R$ be an integral domain with the following property: One has $\{0\}\neq pR\neq R$ and for every isomorphism $\phi\colon P\liso Q$ in the category $\calF$ and every prime divisor $q$ of $[\Aut_{\calF}(P): (\Aut_{\calF}(P)\cap\Aut_S(Q)^\phi)]$ one has $\{0\}\neq qR\neq R$. 
Then the center of the ring $RB^{\calF}(S,S)$ is connected. In particular, when $R=\ZZ$, the center of $B^{\calF}(S,S)$ is connected.
\end{theorem}

\begin{proof} 
The proof is similar to the proof of Theorem~\ref{thm RBDelta}. Let $K$ denote the field of fractions of $R$. We will use the diagram
\begin{diagram}
  \movevertex(-50,0){RB(S)} & \movearrow(-50,0){\Ear{\Delta}}  & \movevertex(-50,0){RB^{\calF}(S,S)} & \movearrow(-50,0){\Ear[30]{\sigmatilde_S^{\calF}}} & 
         \movevertex(0,-10){\mathop{\prod}\limits_{P\in\Sigmahat^{\calF}_S} 
         \End_{R\Out_{\calF}(P)}(R\Hombar_{\calF}(P,S))}&&
  \movearrow(-50,0){\sar} &  & \movearrow(-50,0){\sar} &  & \movearrow(0,0){\sar} &&
  \movevertex(-50,0){KB(S)} & \movearrow(-50,0){\Ear{\Delta}} & \movevertex(-50,0){KB^{\calF}(S,S)} &
        \movearrow(-50,0){\Ear[30]{\sigmatilde_S^{\calF}}} &
     \movevertex(0,-10){\mathop{\prod}\limits_{P\in\Sigmahat^{\calF}_S} 
     \End_{K\Out_{\calF}(P)}(K\Hombar_{\calF}(P,S))} &&
\end{diagram}
Again, each map in the diagram is injective and the map $\sigmatilde_S^{\calF}$ of the bottom row is an isomorphism, since $\mathrm{char}(K)\neq p$.

Let $e$ be a non-zero central idempotent of $RB^{\calF}(S,S)$. We will show that $e=1$.
The element $\sigmatilde_S^{\calF}(e)$ is a central idempotent in $\prod_{P\in\Sigmahat^{\calF}_S} \End_{R\Out_{\calF}(P)}(R\Hombar_{\calF}(P,S))$. We want to invoke Proposition~\ref{prop idempotents of Hecke algebras} and need to determine the stabilizer of $[\phi]$ for $\phi\in\Hom_{\calF}(P,S)$. It follows from an easy calculation that $\stab_{\Aut_{\calF}(P)}([\phi])=\Aut_{\calF}(P)\cap \Aut_S(Q)^\phi$, where $Q:=\phi(P)$. Thus, by Proposition~\ref{prop idempotents of Hecke algebras}, there exists  a subset $\Xi\subseteq\Sigma_S$ which is closed under taking $\calF$-isomorphic subgroups, such that, with $\Xihat:=\Xi\cap\Sigmahat_S^{\calF}$, one has:
\begin{equation*}
  \sigmatilde_S^{\calF}(e) = (\delta_{P\in\Xihat})_{P\in\Sigmahat_S^{\calF}}\,,
\end{equation*}
where $\delta_{P\in\Xihat}$ denotes the identity map if $P\in\Xihat$ and the zero map otherwise.
Now Proposition~\ref{prop image of e_U for F} implies that 
\begin{equation*}
  \sigmatilde_S^{\calF}(e)=\sigmatilde_S^{\calF}(\Delta(\sum_{P\in\Xitilde} e_P))\,,
\end{equation*}
where $\Xitilde:=\Xi\cap\Sigmatilde_S$. By the injectivity of $\sigmatilde_S^{\calF}$, we have
\begin{equation*}
  e=\Delta(\sum_{P\in\Xitilde} e_P) \in RB^{\calF}(S,S)\cap \Delta(KB(S)) = \Delta(RB(S))\,.
\end{equation*}
Thus, the idempotent $\sum_{P\in\Xitilde} e_P$ of $KB(S)$ is contained in $RB(S)$. 
But since $pR\neq R$, this implies that $\Xi=\Sigma_S$ and that $\sum_{P\in\Xitilde} e_P = 1$, by Remark~\ref{rem idempotents in B(G)}.
\end{proof}


\section{Central idempotents of $RB^{\trl}(G,G)$}\label{sec left-free}

Throughout this section, $R$ denotes an integral domain, $K$ denotes a field of characteristic $0$, $G$ denotes a finite group, and $\Sigmahat_G\subseteq\Sigmatilde_G\subseteq\Sigma_G$ denote sets of representatives of isomorphism classes, resp. representatives of conjugacy classes in the set $\Sigma_G$ of subgroups of $G$.

The goal of this section is to show that the Grothendieck ring $B^{\trl}(G,G)$ of left-free $(G,G)$-bisets has no central idempotent different from $0$ and $1$. We will prove the same result for a class of scalar extensions $RB^{\trl}(G,G)$ from $\ZZ$ to $R$ for certain integral domains $R$ (see Theorem~\ref{thm left-free}). We will also give a parametrizing set of the blocks of $KB^{\trl}(G,G)$ for fields $K$ of characteristic $0$ (see Corollary~\ref{cor general}). The description of the parametrizing set in this corollary is made more explicit in character-theoretic terms in the key lemma~\ref{lem eBtrle'}. First, we will use the following result, see Theorem~6.4(c) from \cite{BD}.

\begin{proposition}\label{prop radical decomp}
Let $K$ be a field of characteristic $0$. Then one has a decomposition $KB^{\trl}(G,G) = KB^\Delta(G,G)\oplus J$, where $J$ denotes the Jacobson radical of $KB^{\trl}(G,G)$.
\end{proposition}

By the above proposition, the following lemma will apply to $KB^{\trl}(G,G)$ and its subalgebra $KB^\Delta(G,G)$, for fields $K$ of characteristic $0$. We denote the Jacobson radical of a ring $\Lambda$ by $J(\Lambda)$.

\begin{lemma}\label{lem idempotent and decomp}
Let $\Lambda$ be a ring and let $\Gamma$ be a (not necessarily unitary) subring of $\Lambda$ such that $\Lambda=\Gamma\oplus J(\Lambda)$. Then each central idempotent of $\Lambda$ is contained in $\Gamma$.
\end{lemma}

\begin{proof}
Let $e\in Z(\Lambda)$ be an idempotent and write $e=f+x$ with $f\in\Gamma$ and $x\in J(\Lambda)$. Then $f^2\equiv e^2=e\equiv f \mod J(\Lambda)$ implies $f^2-f\in \Gamma\cap J(\Lambda)$ so that $f$ is an idempotent. Now also $e-ef=ex$ and $f-ef=-xf$ are idempotents and contained in $J(\Lambda)$. Thus, $e-ef=0=f-ef$ and $e=f\in\Gamma$.
\end{proof}

Proposition~\ref{prop radical decomp} and Lemma~\ref{lem idempotent and decomp} imply the following corollary.

\begin{corollary}\label{cor left-free idempotents over field}
Let $K$ be a field of characteristic $0$. Then every central idempotent of $KB^{\trl}(G,G)$ is already contained in $KB^\Delta(G,G)$.
\end{corollary}

In order to determine the primitive central idempotents  of $KB^{\trl}(G,G)$ the following lemma will be useful.

\begin{lemma}\label{lem central idempotents via relation}
Let $\Lambda$ be a ring and assume that $1=\sum_{i\in I} e_i$ is a decomposition of $1\in\Lambda$ into a finite sum of non-zero pairwise orthogonal (not necessarily primitive) idempotents of $\Lambda$ with the property that for each central idempotent $f$ of $\Lambda$ and each $i\in I$ one has $e_if\in\{e_i,0\}$. Denote by $\sim$ the symmetric and reflexive relation on $I$ defined by $i\sim j$ if and only if $e_i\Lambda e_j\neq 0$ or $e_j\Lambda e_i\neq 0$, and denote by $\approx$ the transitive closure of $\sim$; that is, $i\approx j$ if and only if there exists a sequence $i=i_0, i_1,\ldots, i_n=j$ in $I$ such that $i_{k-1}\sim i_k$ for all $k=1,\ldots,n$. Then $\approx$ is an equivalence relation. If $I_1,\ldots, I_s$ denote the equivalence classes of $I$ with respect to $\approx$ then the elements $f_k:=\sum_{i\in I_k} e_i$, $k=1,\ldots,s$, are primitive pairwise orthogonal central idempotents of $\Lambda$ with $f_1+\cdots+f_s=1$. 
\end{lemma}

\begin{proof}
All statements in the lemma, except for the last sentence, clearly hold. It is also clear that the elements $f_1,\ldots, f_s$ are pairwise orthogonal idempotents whose sum is equal to $1$.

We show first that $f_k\in Z(\Lambda)$ for all $k=1,\ldots,s$. In fact, for $x\in \Lambda$ we have
\begin{equation*}
  xf_k = 1xf_k = \sum_{i\in I}\sum_{j\in I_k} e_ixe_j = \sum_{i,j\in I_k} e_i x e_j\,,
\end{equation*}
since $e_ixe_j\in e_i\Lambda e_j=0$ whenever $i\in I\smallsetminus I_k$ and $j\in I_k$. Similarly, one has $f_kx=\sum_{i,j\in I_k} e_i x e_j$ and therefore $xf_k=f_kx$, and $f_k\in Z(\Lambda)$.

Next we show that, for each $k= 1,\ldots,s$, the central idempotent $f_k$ is primitive in $Z(\Lambda)$. Assume that $f_k = g + h$ is an orthogonal decomposition with central idempotents $g$ and $h$, and assume that $g\neq 0$. Then $0\neq g = gf_k$ implies that $ge_i\neq 0$ for some $i\in I_k$ and therefore $ge_i=e_i$. But then, for each $j\in I$ with $e_i\Lambda e_j\neq 0$ one has $0\neq e_i\Lambda e_j= e_ig\Lambda e_j = e_i\Lambda ge_j$. This implies that $ge_j\neq 0$ and therefore $ge_j=e_j$. Similarly, also $e_j\Lambda e_i\neq 0$ implies that $ge_j = e_j$. Thus, we obtain $ge_j=e_j$ for all $j\in I_k$. This implies $g= gf_k=\sum_{j\in I_k} ge_j = \sum_{j\in I_k}e_j=f_k$ and $h=0$. Thus, $f_k$ is a primitive central idempotent. 
\end{proof}

\begin{corollary}\label{cor general}
Let $K$ be a field of characteristic $0$. Then the primitive central idempotents of $KB^{\trl}(G,G)$ are parametrized by the equivalence classes of $\calE_G$ under the equivalence relation $\approx$ defined as the transitive closure of the relation $\sim$ which is defined by
\begin{equation*}
  (U,\chi)\sim(U',\chi'):\iff e_{(U,\chi)}\cdotG KB^{\trl}(G,G)\cdotG e_{(U',\chi')}\neq 0 \quad\text{or}\quad 
  e_{(U',\chi')}\cdotG KB^{\trl}(G,G)\cdotG e_{(U,\chi)} \neq 0\,.
\end{equation*}
If $\calE\subseteq\calE_G$ is such an equivalence class then $\sum_{(U,\chi)\in\calE}e_{(U,\chi)}$ is the corresponding primitive central idempotent.
\end{corollary}

\begin{proof}
This follows immediately from Lemma~\ref{lem central idempotents via relation} applied to the idempotents $e_{(U,\chi)}$ of $KB^{\trl}(G,G)$. They satisfy the hypothesis of the Lemma, since they are the primitive central idempotents of $KB^\Delta(G,G)$ (see Remark~\ref{rem bifree idempotents over field}) and since each central idempotent of $KB^{\trl}(G,G)$ is contained in $KB^{\Delta}(G,G)$ (see Corollary~\ref{cor left-free idempotents over field}).
\end{proof}

For a subset $X$ of a finite group $G$, a field $K$ of characteristic $0$, and a character $\chi$ of a $KG$-module, we set $X^+:=\sum_{x\in X}x\in KG$ and $\chi(X^+):=\sum_{x\in X}\chi(x)$. By $\chi^*$ we denote the contragredient character of $\chi$. For a subgroup $V$ of $G$, we denote by $\Aut_G(V)$ the image of the map $N_G(V)\to\Aut(V)$, $g\mapsto c_g$, and by $\Out_G(V)$ the image of $\Aut_G(V)$ under the canonical epimorphism $\Aut(V)\to\Out(V)$. If $\chi$ is a $K$-character of $\Aut(U)$ for some finite group $U$ and if $V$ is another group that is isomorphic to $U$ then $\chi_V$ denotes the character of $\Aut(V)$ defined by $\chi_V(\omega):=\chi(\lambda^{-1}\circ\omega\circ\lambda)$ for any isomorphism $\lambda\colon U\myiso V$. The character $\chi_V$ is independent of the choice of $\lambda$. Similarly, if $\chi$ is a $K$-character of $\Out(U)$ and $V$ is isomorphic to $U$ one defines the character $\chi_V$ of $\Out(V)$. Recall the definition of $\Sigmatilde_G(U)$ for $U\in\Sigmahat_G$ from Remark~\ref{rem bifree idempotents over field}.

\begin{lemma}\label{lem eBtrle'}
Let $K$ be a field of characteristic $0$. 

\smallskip
{\rm (a)} Let $\chi\in\Irr_K(\Out(U))$. The pair $(U,\chi)$ belongs to $\calE_G$ if and only if there exists $V\in\Sigmatilde_G(U)$ such that $(\chi_V|_{\Out_G(V)},1)\neq 0$.

\smallskip
{\rm (b)} For any $(U,\chi),(U',\chi')\in\calE_G$, one has $e_{(U,\chi)}\cdotG KB^{\trl}(G,G)\cdotG e_{(U',\chi')}\neq\{0\}$ if and only if there exist $V\in\Sigmatilde_G(U)$, $V'\in\Sigmatilde_G(U')$, and an epimorphism $\alpha\colon V'\to V$ such that 
\begin{equation}\label{eqn main cond}
  (\chi_V^*\times\chi'_{V'})
  \bigl([(\Out_G(V)\times\Out_G(V')) \cdot \Lbar_\alpha]^+\bigr) \neq 0\,.
\end{equation}
Here, $L_\alpha:=\stab_{\Aut(V)\times\Aut(V')}(\alpha)$ under the action $(\omega,\omega')\cdot\alpha:=\omega\circ\alpha\circ(\omega')^{-1}$ and $\Lbar_{\alpha}\le \Out(V)\times \Out(V')$ denotes the image of $L_\alpha$ under the canonical epimorphism $\Aut(V)\times \Aut(V')\to\Out(V)\times\Out(V')$. 
\end{lemma}

The proof of Lemma~\ref{lem eBtrle'} is very technical and will be given in Section~\ref{sec proof of lemma}.

\begin{corollary}\label{cor trivial subgroup}
Let $K$ be a field of characteristic $0$. Then, for any $(U,\chi)\in\calE_G$, one has:
\begin{equation*} 
  e_{(1,1)} \cdotG KB^{\trl}(G,G) \cdotG e_{(U,\chi)} \neq \{0\} \iff \chi=1\,.
\end{equation*}
\end{corollary}

\begin{proof}
By Lemma~\ref{lem eBtrle'}, the condition $e_{(1,1)} KB^{\trl}(G,G) e_{(U,\chi)}\neq\{0\}$ is equivalent to the existence of $V\in\Sigmatilde_G(U)$ such that 
\begin{equation*}
  (1\times \chi_{V})\bigl( [(\Out_G(1)\times\Out_G(V))\cdot \Lbar_\alpha]^+\bigr)
  \neq 0\,.
\end{equation*}
Note that here $\alpha \colon V\to 1$ is the trivial homomorphism and that consequently $L_\alpha=\Aut(1)\times\Aut(V)$. Identifying $\Aut(1)\times\Aut(V)$ with $\Aut(V)$ and $\Out(1)\times\Out(V)$ with $\Out(V)$ via the second projection, we obtain
\begin{align*}
  & e_{(1,1)} \cdotG KB^{\trl}(G,G) \cdotG e_{(U,\chi)}\neq\{0\}
        \iff \chi_{V}\bigl((\Out_G(V)\cdot \Out(V))^+\bigr) \neq 0 \\
  \iff & \chi_{V}\bigl( \Out(V)^+\bigr) \neq 0 \iff (\chi_{V},1)\neq 0 \iff \chi_{V} = 1 \iff \chi = 1\,,
\end{align*}
and the proof is complete.
\end{proof}

\begin{remark}\label{rem L_alpha}
Assume that $\alpha\colon V'\to V$ is a surjective group homomorphism. We want to get a better understanding of the subgroup $L_\alpha$ of $\Aut(V)\times\Aut(V')$ in Lemma~\ref{lem eBtrle'}. Set $\Aut(V',\ker(\alpha)):=\{\omega'\in\Aut(V')\mid \omega'(\ker(\alpha))=\ker(\alpha)\}$. Then $\alpha$ induces a group homomorphism
\begin{equation*}
  \alpha_*\colon \Aut(V',\ker(\alpha))\to \Aut(V)
\end{equation*}
where $(\alpha_*(\omega'))(\alpha(v')):=\alpha(\omega'(v'))$ for $\omega'\in\Aut(V',\ker(\alpha))$ and $v'\in V'$. It is now straightforward to verify that $p_1(L_\alpha)=\im(\alpha_*)$, $k_1(L_\alpha)=\{1\}$, $p_2(L_\alpha)=\Aut(V',\ker(\alpha))$, $k_2(L_\alpha):=\ker(\alpha_*)$ and that the isomorphism $p_2(L_\alpha)/k_2(L_\alpha)\to p_1(L_\alpha)/k_1(L_\alpha)$ determined by $L_\alpha$ is equal to the isomorphism $\alphabar_*\colon \Aut(V',\ker(\alpha))/\ker(\alpha_*)\to \im(\alpha_*)$. Note that $p_1(L_\alpha)$ consists of all automorphisms $\omega$ of $V$ which can be \lq\lq lifted\rq\rq\ (via $\alpha$) to an automorphism $\omega'$ of $V'$, i.e., $\omega\alpha=\alpha\omega'\colon V'\to V$. 

Moreover, $\Inn(V')\le \Aut(V',\ker(\alpha))$, since $\ker(\alpha)$ is normal in $V'$, 
and $\alpha_*(c_{v'})= c_{\alpha(v')}$ for $v'\in V'$, so that $\alpha_*(\Inn(V'))=\Inn(V)$. This implies that the subgroup 
$\Lbar_\alpha=L_\alpha\cdot(\Inn(V)\times\Inn(V'))/\left(\Inn(V)\times\Inn(V')\right)$ of $\Out(V)\times\Out(V')$ satisfies
$p_1(\Lbar_\alpha)=\im(\alpha_*)/\Inn(V)$, $k_1(\Lbar_\alpha)=1$, $p_2(\Lbar_\alpha)=\Aut(V,\ker(\alpha))/\Inn(V')$, $k_2(\Lbar_{\alpha})= \ker(\alpha_*)\cdot\Inn(V')/\Inn(V')$ and the isomorphism $p_2(\Lbar_{\alpha})/k_2(\Lbar_\alpha)\to p_1(\Lbar_\alpha)/k_1(\Lbar_\alpha)$ corresponding to $\Lbar_\alpha$ is induced by $\alpha_*$ and again denoted by $\alphabar_*$.\end{remark}

We refer the reader to \cite[Chapter 2]{BoucSLN} for the definitions of deflation maps $\defl_{G/N}^G\colon R_K(G)\to R_K(G/N)$ 
and inflation maps $\infl_{G/N}^G\colon R_K(G/N)\to R_K(G)$ when $N\trianglelefteq G$, and the isomorphism maps $\isom_{\alpha}\colon R_K(G_2)\to R_K(G_1)$ when $\alpha\colon G_1\myiso G_2$ is an isomorphism. Here $K$ is a field of characteristic $0$ and $R_K(G)$ denotes the group of virtual $K$-characters, i.e., the $\ZZ$-span of $\Irr_K(G)$. More generally one also defines inflation and deflation maps for arbitrary epimorphisms by combining the above definitions with an isomorphism map.

\begin{lemma}\label{lem trivial Out_G}
Let $K$ be a field of characteristic $0$ and let $\alpha\colon V'\to V$ be a surjective group homomorphism between subgroups of $G$, and let $\chi\in\Irr_K(\Out(V))$ and $\chi'\in\Irr_K(\Out(V'))$. Assume further that $\Out_G(V)$ and $\Out_G(V')$ are trivial. Then the following are equivalent:

\smallskip
{\rm (i)} The condition in (\ref{eqn main cond}) holds with $\chi_V$ and $\chi'_{V'}$ replaced by $\chi$ and $\chi'$, respectively.

\smallskip
{\rm (ii)} One has $((\chi^*\times\chi')|_{\Lbar_\alpha},1)\neq 0$.

\smallskip
{\rm (iii)} The irreducible character $\chi$ is a constituent of the image of $\chi'$ under the composition of the following sequence of maps: $\res^{\Out(V')}_{\Out(V',\ker(\alpha))}$, $\defl^{\Out(V',\ker(\alpha))}_{\Out(V',\ker(\alpha))/k_2(\Lbar_\alpha)}$, $\isom_{\alphabar_*}$, $\ind_{\im(\alpha_*)/\Inn(V)}^{\Out(V)}$.

\smallskip
{\rm (iv)} The irreducible character $\chi'$ is a constituent of the image of $\chi$ under the composition of the following sequence of maps: $\res^{\Out(V)}_{\im(\alpha_*)/\Inn(V)}$, $\isom_{\alphabar_*}^{-1}$, $\infl_{\Out(V',\ker(\alpha))/k_2(\Lbar_{\alpha})}^{\Out(V',\ker(\alpha))}$, $\ind_{\Out(V',\ker(\alpha))}^{\Out(V')}$.
\end{lemma}

\begin{proof}
Clearly, (i) and (ii) are equivalent. 

The equivalence between (ii) and (iii) now follows from the following general consideration:
If $A$ and $B$ are finite groups such that $|A\times B|$ is invertible in $K$ and if $L$ is a subgroup of $A\times B$ then the permutation character of $A\times B/L$ is equal to the sum of the characters $I(\psi)\times\psi^*$, where $\psi$ runs through the irreducible characters of $B$ and $I\colon R_K(B)\to R_K(A)$ is the map induced by tensoring with the $(KG,KH)$-bimodule $K[A\times B/L]$ over $KB$. Now the result follows from the decomposition of the transitive biset $A\times B/L$ as in \cite[Lemma~2.3.26]{BoucSLN}. 

Finally, the equivalence between (ii) and (iv) follows from the last equivalence applied to the dual subgroup $L^\circ:=\{(b,a)\in B\times A\mid (a,b)\in L\}$ of $L$.
\end{proof}

\begin{corollary}
Let $K$ be a field of characteristic $0$. Then $(G,\chi)\in\calE_G$ for each $\chi\in\Irr_K(\Out(G))$. Moreover, for any $\chi,\chi'\in\Irr_K(\Out(G))$ one has
\begin{equation*}
  e_{(G,\chi)}\cdotG KB^{\trl}(G,G)\cdotG e_{(G,\chi')} \neq \{0\} \iff \chi=\chi'\,.
\end{equation*}
\end{corollary}

\begin{proof}
Note that $\Out_G(G)$ is the trivial subgroup of $\Out(G)$. Thus, Lemma~\ref{lem eBtrle'}(a) implies the first statement. Now let $\chi,\chi'\in\Irr_K(\Out(G))$. Since $\Out_G(G)$ is trivial we can use Lemma~\ref{lem trivial Out_G}. First note that if $\alpha\in\Aut(G)$ then $L_\alpha=\lexp{(\alpha,1)}{\Delta(\Aut(G))}$ and $\Lbar_\alpha\le\Out(G)\times\Out(G)$ is conjugate to $\Delta(\Out(G))$. Thus, the composition of the sequence of the maps in (iv) in Lemma~\ref{lem trivial Out_G} is the identity map, and the result follows.
\end{proof}

\begin{theorem}\label{thm left-free}
Let $R$ be an integral domain with field of fractions $K$ of characteristic $0$. Assume further that for  every $U\in\Sigmahat_G$ and for every prime divisor $p$ of $|\Out(U)|$ one has $pR\neq R$. Then $RB^{\trl}(G,G)$ has no central idempotent different from $0$ and $1$. In particular, the ring $B^{\trl}(G,G)$ is connected.
\end{theorem}

\begin{proof}
Let $e$ be a non-zero central idempotent of $RB^{\trl}(G,G)$. We will show that $e=1$. We will use the commutative diagram of canonical embeddings
\begin{diagram}[90]
  RB^\Delta(G,G) & \subseteq & RB^{\trl}(G,G) &&
  \sar & & \sar &&
  KB^\Delta(G,G) & \subseteq & KB^{\trl}(G,G) &&
\end{diagram}
Since $e$ is a central idempotent of $RB^{\trl}(G,G)$, it is also a central idempotent of $KB^{\trl}(G,G)$. By Corollary~\ref{cor left-free idempotents over field}, we obtain that $e$ is a central idempotent of $KB^\Delta(G,G)$. From Remark~\ref{rem bifree idempotents over field}(a) we obtain that
\begin{equation*}
  e=\sum_{(U,\chi)\in\calE} e_{(U,\chi)}
\end{equation*}
for a subset $\calE$ of $\calE_G$. By Lemma~\ref{lem central idempotents via relation} applied to the ring $KB^{\trl}(G,G)$ and the idempotents $e_{(U,\chi)}$, $(U,\chi)\in\calE_G$, we know that the subset $\calE$ has the property that if $(U,\chi)\in\calE$ and $(U',\chi')\in\calE_G$ satisfy $e_{(U,\chi)}\cdotG KB^{\trl}(G,G)\cdotG e_{(U',\chi')}\neq\{0\}$ or $e_{(U',\chi')}\cdotG KB^{\trl}(G,G)\cdotG e_{(U,\chi)}\neq \{0\}$ then also $(U',\chi')\in\calE$. However, Corollary~\ref{cor trivial subgroup} implies:
\begin{equation}\label{eqn first cond}
  \text{If $(U,1)\in\calE$ for some $U\in\Sigmahat_G$ then $(U',1)\in\calE$ for all $U'\in\Sigmahat_G$.}
\end{equation}
Since $KB^\Delta(G,G)\cap RB^{\trl}(G,G)=RB^\Delta(G,G)$, we also obtain that $e$ is a central idempotent of $RB^\Delta(G,G)$. Now Theorem~\ref{thm RBDelta} and Proposition~\ref{prop image of e_U} imply:
\begin{equation}\label{eqn second cond}
  \text{If $(U,\chi)\in\calE$ then $(U,\chi')\in\calE$ for all $\chi'\in\Irr_K(\Out(U))$.}
\end{equation}
Since $e\neq 0$, there exists at least one element $(U,\chi)\in \calE$. But then (\ref{eqn first cond}) and (\ref{eqn second cond}) together imply that $\calE=\calE_G$, or in other words that $e=1$.
\end{proof}

Parts (a) and (b) of the next proposition will be used in Section~\ref{sec proof of lemma} to prove Lemma~\ref{lem eBtrle'}. Part~(c) is a surprising fact on the number in condition (\ref{eqn main cond}). We are grateful to John Murray for bringing Stembridge's paper \cite{S} to our attention. It is used in the proof of the proposition.

\begin{proposition}\label{prop stembridge}
Let $K$ be a field of characteristic $0$ and let $A$ be a finite group. Moreover, let $B$ and $C$ be subgroups of $A$, let $a\in A$, and let $\chi\in\Irr_K(A)$.

\smallskip
{\rm (a)} If $\chi((aB)^+)\neq 0$ then $\chi(B^+)\neq 0$ and $(\chi|_B,1)\neq 0$.

\smallskip
{\rm (b)} If $\chi((BaC)^+)\neq 0$ then $\chi((BC)^+)\neq 0$.

\smallskip
{\rm (c)} If $K\subseteq\CC$ then $\chi((BC)^+)$ is a non-negative real number.
\end{proposition}

\begin{proof}
(a) If $\chi((aB)^+)\neq 0$ then $\chi(B^+)\neq 0$ by the proof of \cite[Lemma~7.3]{BD2}. Clearly, $\chi(B^+)\neq 0$ is equivalent to $(\chi|_B,1)\neq 0$.

\smallskip
(b) Assume that $\chi((BC)^+)=0$. Consider the idempotents $e_B:=\frac{1}{|B|}\sum_{b\in B}b$ and $e_C:=\frac{1}{|C|} \sum_{c\in C}c$ of $KA$. Then also
\begin{equation*}
  \chi(e_Be_C) = \frac{1}{|B|\cdot|C|}\, \chi(B^+\cdot C^+) = \frac{|B\cap C|}{|B|\cdot|C|}\, \chi((BC)^+) = 0\,.
\end{equation*}
Using isomorphisms between suitable fields we may assume that $\chi$ is the character of a matrix representation $\rho\colon A\to\GL_n(\CC)$ and we can assume that $\rho$ is unitary. We follow the proof of \cite[Lemma~1.2]{S}. Writing $M^*$ for the conjugate transposed of a complex matrix $M$, we obtain $\rho(e_Be_C)^*=\rho(e_C)^*\rho(e_B)^*=\rho(e_C)\rho(e_B)=\rho(e_Ce_B)$. This implies
\begin{equation*}
  0 = \chi(e_Be_C)= \chi(e_B^2e_C^2) = \chi(e_Be_C^2e_B)
  = \tr(\rho(e_Be_C)\rho(e_Ce_B)) = \tr(\rho(e_Be_C)\rho(e_Be_C)^*)\,.
\end{equation*}
Thus, the square matrix $M:=\rho(e_Be_C)$ satisfies $\tr(MM^*)=0$. However, $MM^*$ is hermitian and therefore diagonalizable, and each eigenvalue of $MM^*$ is real and non-negative. This implies that $MM^*=0$, which in turn implies that $M=0$. Now we obtain 
\begin{equation*}
  \frac{|B\cap aCa^{-1}|}{|B|\cdot|C|}\ \chi((BaC)^+)= \chi(e_Bae_C)=\chi(ae_Ce_B)= \tr(\rho(a)\rho(e_Ce_B))=0\,,
\end{equation*}
a contradiction. 

\smallskip
(c) Since $\chi((BC)^+)=\frac{|B|\cdot|C|}{|B\cap C|}\chi(e_Be_C)$, this follows immediately from \cite[Lemma~1.2]{S}. 
\end{proof}


\section{Proof of Lemma~\ref{lem eBtrle'}}\label{sec proof of lemma}

The goal of this section is the proof of Lemma~\ref{lem eBtrle'}. Part~(a) of Lemma~\ref{lem eBtrle'} follows immediately from the last statement in Proposition~\ref{prop rho(e)} (using the definition of $\calE_G$ from Remark~\ref{rem bifree idempotents over field}) and part (b) will be proved at the end of this section.

Throughout this section, $G$ denotes a finite group and $K$ denotes a field of characteristic $0$. We first need to recall constructions from \cite{BD}. 

\begin{remark}\label{rem tau}
(a) Recall from \cite[Section~4]{BD} that there is an isomorphism
\begin{equation}\label{eqn rho}
  \rho_G\colon KB^{\trl}(G,G)\to {KA(G,G)}^{G\times G}\,,
\end{equation}
where $A(G,G)$ is the free abelian group on the set of triples $(U,\alpha,V)$, where $U$ and $V$ are subgroups of $G$ and $\alpha\colon V\to U$ is an epimorphism.  Moreover $KA(G,G)$ denotes the $K$-vector space with the same triples as basis. The group $G\times G$ acts on these triples by $\lexp{(g,h)}{(U,\alpha,V)}:=(\lexp{g}{U},c_g\alpha c_h^{-1},\lexp{h}{V})$, and $KA(G,G)^{G\times G}$ denotes the fixed points under the extended action on $KA(G,G)$. The $G\times G$-orbit of $(U,\alpha,V)$ is denoted by $[U,\alpha,V]_{G\times G}$ and the class sums $[U,\alpha,V]_{G\times G}^+$ form a $K$-basis of $KA(G,G)^{G\times G}$. By \cite[Theorem~4.7]{BD}, the multiplication in $KB^{\trl}(G,G)$ is translated under the isomorphism $\rho_G$ into the multiplication on $KA(G,G)^{G\times G}$ which is the restriction of the following multiplication on $KA(G,G)$:
\begin{equation*}
  (U,\alpha,V)\cdotG (V',\beta,W):=\begin{cases} \frac{|C_G(V)|}{|G|} (U,\alpha\circ\beta, W) & \text{if $V=V'$,} \\
                                                                             0 & \text{if $V\neq V'$.} 
                                                       \end{cases}
\end{equation*}

\smallskip
(b) Under the isomorphism $\rho_G$ in (a), the subspace $KB^\Delta(G,G)$ is mapped isomorphically onto the $G\times G$-fixed points of the $K$-span $KA^\Delta(G,G)$ of triples of the form $(U,\alpha,V)$, where $\alpha\colon V\myiso U$ is an isomorphism. The isomorphism
\begin{equation*}
  \rho_G\circ \sigma_G^{-1} \colon \bigoplus_{U\in\Sigmahat_G}\End_{K\Out(U)}(K\Injbar(U,G)) 
  \liso KA^\Delta(G,G)^{G\times G}
\end{equation*}
is given explicitly as follows: For $U\in\Sigmahat_G$ and a $K\Out(U)$-module endomorphism $f\colon [\mu]\mapsto \sum_{[\lambda]} a_{[\lambda],[\mu]} [\lambda]$ of $K\Injbar(U,G)$, the image of $f$ under $\rho_G\circ \sigma_G^{-1}$ is given by
\begin{equation*}
  \sum_{\lambda\times_{\Aut(U)} \mu \in \Inj(U,G)\times_{\Aut(U)}\Inj(U,G)} 
  a_{[\lambda],[\mu]} \cdot (\lambda(U), \lambda\circ\mu^{-1},\mu(U))\,,
\end{equation*}
cf.~the proof of Theorem~5.5(d) in \cite{BD}. Here, $\Aut(U)$ acts on $\Inj(U,G)$ from the right via composition and from the left by using the right action and inversion of group elements.
\end{remark}

For some of the notation in the following proposition we refer the reader back to Remark~\ref{rem bifree idempotents over field}.

\begin{proposition}\label{prop rho(e)}
Let $(U,\chi)\in\calE_G$, $V\in\Sigmatilde_G(U)$, and let $\chi=s(\psi_1+\cdots+\psi_r)$ be a decomposition of $\chi$ into absolutely irreducible characters over some extension field of $K$. Then
\begin{equation}\label{eqn rho(e)}
  \rho_G(e_{(U,\chi,V)}) = \frac{\chi(1)}{s^2 r |\Out(U)|} \sum_{\substack{g,h\in G/N_G(V) \\ \omega\in\Aut(V)}}
  \chi_V\bigl((\omegabar^{-1} \cdot \Out_G(V))^+\bigr) \cdot \lexp{(g,h)}{(V,\omega,V)}\,.
\end{equation}
Moreover, $e_{(U,\chi,V)}\neq 0$ if and only if $(\chi_V|_{\Out_G(V)},1)\neq 0$. In particular, Lemma~\ref{lem eBtrle'}(a) holds.
\end{proposition}

\begin{proof}
We set $c:=\chi(1)/(s^2r|\Aut(U)|)$. Using the explicit formula (\ref{eqn e_chi}) for $e_\chi$ in Remark~\ref{rem bifree idempotents over field}, the $U$-component $f$ of $\sigma_G(e_{(U,\chi,V)})$ is given by
\begin{equation*}
  [\mu]\mapsto c
  \sum_{\omega\in\Aut(U)} \chi(\omega^{-1}) [\mu\omega]\,,
\end{equation*}
for $\mu\in\Inj(U,G)_V$, if we denote the inflation of $\chi$ to $\Aut(U)$ again by $\chi$. Thus, the matrix coefficients $a_{[\lambda],[\mu]}$ of $f$ with respect to the basis $\Injbar(U,G)_V$ are given by 
\begin{equation*}
  a_{[\lambda],[\mu]}= 
  \begin{cases}
     c \mathop{\sum}\limits_{\substack{\omega\in\Aut(U) \\ [\lambda]=[\mu\omega]}}
                 \chi(\omega^{-1}) & \text{if $\mu\in\Inj(U,G)_V$,}\\
     0 & \text{otherwise.}     
  \end{cases}
\end{equation*}   
Using the explicit description of $\rho_G\circ\sigma_G^{-1}$ in Remark~\ref{rem tau}(b), we obtain
\begin{equation*}
  \rho_G(e_{(U,\chi,V)}) = \sum_{\substack{\lambda\times_{\Aut(U)}\mu\in\Inj(U,G)_V\times_{\Aut(U)}\Inj(U,G)_V}}
  \ \sum_{\substack{\omega\in\Aut(U) \\ [\lambda]=[\mu\omega]}}
  c\chi(\omega^{-1}) \cdot (\lambda(U),\lambda\mu^{-1},\mu(U))\,.
\end{equation*}
Note that $\Aut(U)$ acts freely on $\Inj(U,G)_V\times\Inj(U,G)_V$ so that replacing the summation over $\Inj(U,G)_V\times_{\Aut(U)}\Inj(U,G)_V$ yields
\begin{equation*}
  \sum_{(\lambda,\mu)\in\Inj(U,G)_V\times\Inj(U,G)_V}
  \ \sum_{\substack{\omega\in\Aut(U) \\ [\lambda]=[\mu\omega]}}
  c'\chi(\omega^{-1}) \cdot (\lambda(U),\lambda\mu^{-1},\mu(U))\,,
\end{equation*}
where $c':=c/|\Aut(U)|$. Let $\lambda_0\colon U\to V$ be a fixed isomorphism. Note that if $g$ runs through a set of representatives of $G/N_G(V)$ and $\alpha$ runs through $\Aut(U)$ then $c_g\lambda_0\alpha$ runs through $\Inj(U,G)_V$ without repetition. Thus we can rewrite the last expression as
\begin{equation*}
  \sum_{g,h\in G/N_G(V)} \ 
  \sum_{\substack{\alpha,\beta,\omega\in\Aut(U)\\ [c_g\lambda_0\alpha]=[c_h\lambda_0\beta\omega]}}
  c'\chi(\omega^{-1}) \cdot \lexp{(g,h)}{(V,\lambda_0\alpha\beta^{-1}\lambda_0^{-1},V)}\,.
\end{equation*}    
It is straightforward to verify that $[c_g\lambda_0\alpha]=[c_h\lambda_0\beta\omega]$ if and only if $\omega\in\beta^{-1}\lambda_0^{-1}\Aut_G(V)\lambda_0\alpha$. Thus, with $\omega=\beta^{-1}\lambda_0^{-1}\gamma\lambda_0\alpha$ for $\gamma\in\Aut_G(V)$, the last expression can be rewritten as
\begin{equation*}
  \sum_{g,h\in G/N_G(V)} \ 
  \sum_{\alpha,\beta\in\Aut(U)} \ 
  \sum_{\gamma\in\Aut_G(V)} 
  c'\chi(\alpha^{-1}\lambda_0^{-1}\gamma^{-1}\lambda_0\beta) \cdot 
  \lexp{(g,h)}{(V,\lambda_0\alpha\beta^{-1}\lambda_0^{-1},V)}\,.
\end{equation*}
Using $\chi(\alpha^{-1}\lambda_0^{-1}\gamma^{-1}\lambda_0\beta) = 
\chi_V(\lambda_0\beta\alpha^{-1}\lambda_0^{-1}\gamma^{-1})$ and rewriting $\lambda_0\alpha\beta^{-1}\lambda_0^{-1}$ as $\omega\in\Aut(V)$ the triple sum is equal to
\begin{align*}
  & \sum_{g,h\in G/N_G(V)} \ 
  \sum_{\omega\in\Aut(V)} \ 
  \sum_{\gamma\in\Aut_G(V)}  
  c''\chi_V(\omega^{-1}\gamma^{-1}) \cdot 
  \lexp{(g,h)}{(V,\omega,V)} \\
  = & \sum_{\substack{g,h\in G/N_G(V) \\ \omega\in\Aut(V)}} c''\cdot\chi_V((\omega^{-1}\Aut_G(V))^+)\cdot 
  \lexp{(g,h)}(V,\omega,V)
\end{align*}
with $c''=c'\cdot|\Aut(U)|=c$. Since $\chi_V((\omega^{-1}\Aut_G(V))^+) = |\Inn(V)|\cdot  \chi_V((\omegabar^{-1} \cdot \Out_G(V))^+)$, we obtain Equation~(\ref{eqn rho(e)}). 

Finally, note that different choices of triples $(g,h,\omega)$ in the sum of Equation~(\ref{eqn rho(e)}) lead to different basis elements $\lexp{(g,h)}(V,\omega,V)$ of $KA(G,G)$. Thus, $\rho(e_{(U,\chi,V)})\neq 0$ if and only if there exists $\omega\in\Aut(V)$ such that $\chi_V( (\omegabar\cdot\Out_G(V))^+)\neq 0$. By Proposition~\ref{prop stembridge}(a), this is equivalent to $\chi_V(\Out_G(V)^+)\neq 0$. But this in turn is equivalent to $(\chi_V|_{\Out_G(V)},1)\neq 0$.
\end{proof}

\begin{proof}{\bf of Lemma~\ref{lem eBtrle'}(b).}\quad
Let $(U,\chi), (U',\chi')\in\calE_G$. Then the decomposition~(\ref{eqn eUchiV decomp}) implies that
\begin{equation}\label{eqn eBe 1}
  e_{(U,\chi)}\cdotG KB^{\trl}(G,G) \cdotG e_{(U',\chi')} \neq 0
\end{equation}
if and only if there exist $V\in \Sigmatilde_G(U)$ and $V'\in\Sigmatilde_G(U')$ satisfying
\begin{equation}\label{eqn eBe 2}
  e_{(U,\chi,V)}\cdotG KB^{\trl}(G,G) \cdotG e_{(U',\chi',V')} \neq 0\,,
\end{equation}
Applying the isomorphism $\rho_G$ on both sides, the last condition is equivalent to the existence of a basis element $(W,\alpha,W')$ of $KA(G,G)$ such that the $G\times G$-orbit sum $[W,\alpha,W']_{G\times G}^+$ satisfies $\rho_G(e_{(U,\chi,V)})\cdotG [W,\alpha,W']_{G\times G}^+\cdotG \rho_G(e_{(U',\chi',V')}) \neq 0$. Since the last expression is equal to $0$ if $W$ is not $G$-conjugate to $V$ or $W'$ is not $G$-conjugate to $V'$ (by the explicit formula in (\ref{eqn rho(e)})), and since we may replace $(W,\alpha, W')$ by any $G\times G$-conjugate, the condition in (\ref{eqn eBe 2}) is equivalent to the existence of an epimorphism $\alpha\colon V'\to V$ such that
\begin{equation}\label{eqn eBe 3}
  \rho_G(e_{(U,\chi,V)})\cdotG [V,\alpha,V']_{G\times G}^+\cdotG \rho_G(e_{(U',\chi',V')}) \neq 0\,.
\end{equation}
Recall from \cite[Proposition~1.7]{BD} that one has
\begin{equation*}
  [V,\alpha,V']_{G\times G}^+ = \sum_{(x,y)\in\calA\times \calB} \lexp{(x,y)}{(V,\alpha,V')}\,,
\end{equation*}
where $\calA\subseteq G$ is a set of representatives of the cosets $G/k_1(N_{G\times G}(\Delta(V,\alpha, V')))$ and $\calB\subseteq G$ is a set of representatives of the cosets $G/p_2(N_{G\times G}(\Delta(V,\alpha,V')))$. Using the explicit formula~(\ref{eqn rho(e)}), the condition in (\ref{eqn eBe 3}) is equivalent to
\begin{align*}
  \sum_{(x,y)\in\calA\times \calB} \
  \sum_{\substack{g,h\in G/N_G(V) \\ \omega\in\Aut(V)}} \
  \sum_{\substack{g',h'\in G/N_G(V') \\ \omega'\in\Aut(V')}} \
  & \chi_V((\omega^{-1}\Aut_G(V))^+)\chi_{V'}(({\omega'}^{-1}\Aut_G(V'))^+) \cdot \\
  & \cdot\lexp{(g,h)}{(V,\omega,V)}\cdotG \lexp{(x,y)}{(V,\alpha,V')}\cdotG \lexp{(g',h')}{(V',\omega',V')}\neq 0\,.
\end{align*}
Since $k_1(N_{G\times G}(\Delta(V,\alpha,V')))=C_G(V)\le N_G(V)$ and $p_2(N_{G\times G}(\Delta(V,\alpha,V')))\le N_G(V')$, cf.~\cite[Proposition~1.7]{BD}, each element $x\in \calA$ (resp.~$y\in\calB$) determines a unique element $h\in G/N_G(V)$ (resp.~$g'\in G/N_G(V')$) such that the multiplication $\cdotG$ is non-zero, and for given $x\in\calA$ (resp.~$y\in\calB$), we may adjust the representatives $h$ of $G/N_G(V)$ (resp.~$g'$ of $G/N_G(V')$) such that $x$ (resp.~$y$) occurs as a representative. Thus, the above condition is equivalent to
\begin{equation*}
  \sum_{(x,y)\in\calA\times\calB} 
  \sum_{\substack{g\in G/N_G(V) \\ h'\in G/N_G(V')}}
  \sum_{\substack{\omega\in\Aut(V) \\ \omega'\in \Aut(V')}}
  \chi_V((\omega^{-1}\Aut_G(V))^+)\chi'_{V'}(({\omega'}^{-1}\Aut_G(V'))^+) \cdot
  \lexp{(g,h')}{(V,\omega\alpha\omega',V')}\neq 0\,.
\end{equation*}
Since $x$ and $y$ do not occur in the argument of the sum and since $|\calA\times\calB|$ is invertible in $K$, we may drop the first sum in the above condition. Moreover, since for the various choices of $g$ and $h'$ in the above sum, the sets $\{\lexp{(g,h')}{(V,\omega\alpha\omega',V')}\mid (\omega,\omega')\in\Aut(V)\times\Aut(V')\}$ of basis elements are pairwise disjoint, the above condition is also equivalent to
\begin{equation*}
  \sum_{(\omega,\omega')\in\Aut(V)\times\Aut(V')}
  \chi_V((\omega^{-1}\Aut_G(V))^+)
  \chi'_{V'}(({\omega'}^{-1}\Aut_G(V'))^+)\cdot
  (V,\omega\alpha\omega',V')\neq 0\,.
\end{equation*}
Next we fix an element $(\omega_0,\omega'_0)\in\Aut(V)\times\Aut(V')$ and determine the coefficient of $(V,\omega_0\alpha\omega'_0,V')$ in the above sum. Let $L_\alpha:=\stab_{\Aut(V)\times\Aut(V')}(\alpha)$. Then, for any $(\omega,\omega')\in\Aut(V)\times\Aut(V')$, we have
\begin{equation*}
  \omega\alpha\omega' = \omega_0\alpha\omega'_0 \iff 
  (\omega^{-1},\omega')\in L_\alpha (\omega_0^{-1},\omega'_0)\,.
\end{equation*}
Thus, writing $(\omega^{-1},\omega')=(\theta,\theta')(\omega_0^{-1},\omega'_0)$, for $(\theta,\theta')\in L_\alpha$, the last condition is equivalent to requiring that there exists an element $(\omega_0,\omega'_0)\in\Aut(V)\times\Aut(V')$ such that
\begin{equation*}
  \sum_{(\theta,\theta')\in L_\alpha} 
  \chi_V((\theta\omega_0^{-1}\Aut_G(V))^+)
  \chi'_{V'}(({\omega'_0}^{-1}{\theta'_0}^{-1}\Aut_G(V'))^+) \neq 0\,.
\end{equation*}
Since $\chi_V((\theta\omega_0^{-1}\Aut_G(V))^+)=\chi_V^*((\Aut_G(V)\omega_0\theta^{-1})^+)$ and $\chi'_{V'}(({\omega'_0}^{-1}{\theta'_0}^{-1}\Aut_G(V'))^+)= \chi'_{V'}((\Aut_G(V'){\omega'_0}^{-1}{\theta'}^{-1})^+)$, the sum in the above equation is equal to
\begin{align*}
  & \sum_{(\theta,\theta')\in L_\alpha}
  \chi_V^*((\Aut_G(V)\omega_0\theta^{-1})^+)
  \chi'_{V'}((\Aut_G(V'){\omega'_0}^{-1}{\theta'}^{-1})^+) \\
  = \ & c\cdot (\chi_V^*\times\chi'_{V'}) 
  \bigl([(\Aut_G(V)\times\Aut_G(V'))\cdot (\omega_0,{\omega'_0}^{-1})\cdot L_\alpha]^+\bigr)
\end{align*}
with $c=|(\Aut_G(V)\times\Aut_G(V'))\cap \lexp{(\omega_0,{\omega'_0}^{-1})}{L_\alpha}|$. Moreover, 
\begin{align*}
 & (\chi_V^*\times\chi'_{V'}) 
  \bigl([(\Aut_G(V)\times\Aut_G(V'))\cdot (\omega_0,{\omega'_0}^{-1})\cdot L_\alpha]^+\bigr) \\
  = \ & |\Inn(V)\times\Inn(V')|\cdot (\chi_V^*\times\chi'_{V'}) 
  \bigl([(\Out_G(V)\times\Out_G(V'))\cdot (\omegabar_0,{\overline{\omega'_0}}^{-1})\cdot \Lbar_\alpha]^+\bigr)\,.
\end{align*}
In fact, $\Inn(V)\times\Inn(V')$ is contained in $\Aut_G(V)\times\Aut_G(V')$ and the canonical epimorphism $\Aut(V)\times\Aut(V')\to \Out(V)\times\Out(V')$ maps the set $(\Aut_G(V)\times\Aut_G(V'))\cdot (\omega_0,{\omega'_0}^{-1})\cdot L_\alpha$ onto the set $(\Out_G(V)\times\Out_G(V'))\cdot (\omegabar_0,{\overline{\omega'_0}}^{-1})\cdot \Lbar_\alpha$ with fibers of cardinality $|\Inn(V)\times\Inn(V')|$, since $\Inn(V)\times\Inn(V')$ acts freely by left multiplication on the first set. Finally, by Proposition~\ref{prop stembridge}(b), there exists $(\omega_0,\omega'_0)\in\Aut(V)\times\Aut(V')$ such that the right-hand side of the last equation is non-zero if and only if 
\begin{equation*}
  (\chi_V^*\times\chi'_{V'}) \bigl([(\Out_G(V)\times\Out_G(V'))\cdot \Lbar_\alpha]^+\bigr)\neq 0\,.
\end{equation*}
This completes the proof of Lemma~\ref{lem eBtrle'}(b).
\end{proof}


\section{Examples}\label{sec ex}

In this section, for suitable fields $K$, we explicitly parametrize the blocks of $KB^{\trl}(G,G)$ when $G$ is cyclic (see Theorem~\ref{thm cyclic}) and we explicitly determine one particular block of $KB^{\trl}(G,G)$ when $G$ is elementary abelian (see Example~\ref{ex el abelian}). For an abelian group $U$ we will identify $\Aut(U)$ and $\Out(U)$.

\begin{example}
(a) Let $U$ be a cyclic group of order $k$, let $U'$ be a cyclic group of order $k'$, and let $\alpha\colon U'\to U$ be a surjective homomorphism. Then $k$ divides $k'$. We want to determine the subgroup $L_\alpha$ of $\Aut(U)\times \Aut(U')$ using Remark~\ref{rem L_alpha}. First note that $\Aut(U',\ker(\alpha))=\Aut(U')$, since $\ker(\alpha)$ is the only subgroup of order $k'/k$ in $U'$. Also note that one has a canonical isomorphism $(\ZZ/k'\ZZ)^\times\to\Aut(U')$ mapping the residue class of an integer $i$ which is coprime to $k'$ to the automorphism which raises each element to its $i$-th power. Note that if $\omega'\in\Aut(U')$ corresponds to $i$ then $\alpha_*(\omega')\in\Aut(U)$ also corresponds to $i$. Recall that the canonical map $(\ZZ/k'\ZZ)^\times\to(\ZZ/k\ZZ)^\times$ is surjective. Thus, $\alpha_*\colon\Aut(U')\to\Aut(U)$ is the canonical surjective map $p_{k,k'}$ which sends the automorphism $u'\mapsto {u'}^i$ of $U'$ to the automorphism $u\mapsto u^i$ of $U$, for any integer $i$ which is coprime to $k'$. In particular, $\alpha_*$ does not depend on $\alpha$.

\smallskip
(b) Now let $G$ be a cyclic group of order $n$ and let $K$ be a field of characteristic $0$. Assume that $U$ and $U'$ as in (a) are subgroups of $G$ and let $\alpha\colon U'\to U$ be a surjective group homomorphism. Moreover, let $\chi\in\Irr_K(\Aut(U))$ and $\chi'\in\Irr_K(\Aut(U'))$. Note that $(U,\chi)\in\calE_G$, since $\Injbar(U,G)=\Aut(U)$. Note that $\Aut_G(U)$ and $\Out_G(U')$ are trivial so that we can use Lemma~\ref{lem trivial Out_G} which implies that $e_{(U,\chi)}\cdotG KB^{\trl}(G,G)\cdotG e_{(U',\chi')}\neq 0$ if and only if $\chi'=\chi\circ p_{k,k'}$, that is, if and only if $\chi'$ is the inflation of $\chi$ with respect to $p_{k,k'}$. 


\smallskip
(c) We define a partial order $\le$ on the set $\calE_G$ by setting $(U,\chi)\le (U',\chi')$ if and only if $|U|$ divides $|U'|$ and $\chi'=\chi\circ p_{|U|,|U'|}$. Note that the symmetric closure of $\le$ is the same relation as $\sim$ in Corollary~\ref{cor general}. Thus, the connected components (i.e., the equivalence classes of the symmetric and transitive closure of $\le$) are the equivalence classes of $\calE_G$ describing the primitive central idempotents of $KB^{\trl}(G,G)$. We call a pair $(U,\chi)\in\calE_G$ {\em primitive} if it is minimal with respect to $\le$. It is well known that for each element $(U',\chi')\in \calE_G$ there exists a unique primitive element $(U,\chi)$ with $(U,\chi)\le (U',\chi')$. It is now straightforward to see that the equivalence classes of $\calE_G$ are represented by the elements $(G,\vartheta)$, $\vartheta\in\Irr_K(\Aut(G))$, or also by the set of primitive pairs of $\calE_G$.
\end{example}

We summarize the results developed above in the following theorem.

\begin{theorem}\label{thm cyclic}
Let $G$ be a cyclic group of order $n$ and let $K$ be a field of characteristic $0$. Then each pair $(U,\chi)$ with $U\le G$ and $\chi\in\Irr_K(\Aut(U))$ is contained in $\calE_G$. The set of primitive central idempotents of $KB^{\trl}(G,G)$ is parametrized by $\Irr_K(\Aut(G))$. For $\vartheta\in\Irr_K(\Aut(G))$ the corresponding primitive idempotent is the sum of the idempotents $e_{(U,\chi)}$ with $(U,\chi)\in\calE_G$ satisfying $\vartheta=\chi\circ p_{|U|,n}$.
\end{theorem}

\begin{remark}
Let $G$ and $K$ be as in the above theorem. From Theorem~8.11 and Remark~8.12(a) in \cite{BD2} one can see that the primitive central idempotents of the full double Burnside algebra $KB(G,G)$ are indexed by the pairs $(U,\chi)\in\calE_G$. Thus, the primitive central idempotents of $KB^{\trl}(G,G)$ must split in $KB(G,G)$. Note also that the primitive central idempotents of $KB^\Delta(G,G)$ were also indexed by $\calE_G$.
\end{remark}

\begin{example}\label{ex el abelian}
Let $G=\langle x_1,\ldots,x_n\rangle$ be an elementary abelian $p$-group of rank $n$ for a prime $p$ and let $K$ be an algebraically closed field of characteristic $0$. 

\smallskip
(a) For $i=0,\ldots,n$ set $U_i:=\langle x_1,\ldots,x_i\rangle$ and $V_i:=\langle x_{i+1},\ldots, x_n\rangle$, thus $G=U_i\oplus V_i$ as $\FF_p$-vector space. We can choose $\Sigmahat_G$ as $\{U_0,\ldots,U_n\}$. For $0\le i\le j\le n$, let $\pi_{i,j}\colon U_j\to U_i$ denote the canonical projection which is the identity on $U_i$ and has kernel $\langle x_{i+1},\ldots,x_j\rangle=:V_{i,j}$. For $i=0,\ldots,n$ we identify $\Aut(U_i)$ with $G_i:=\GL_i(\FF_p)$ using the basis $(x_1,\ldots, x_i)$ of $U_i$ (with $G_0=\{1\}$). For $0\le i\le j\le n$, the projections $\pi_{i,j}\colon U_j\to U_i$ induce surjections $(\pi_{i,j})_*\colon P_{i,j}:=\Aut(U_j,V_{i,j})\to G_i$, where $P_{i,j}$ is the parabolic subgroup of $G_j$ of shape $\begin{pmatrix}* & 0 \\ * & * \end{pmatrix}$ with upper left corner of size $i\times i$ and lower right corner of size $(j-i)\times (j-i)$. Moreover, we set $Q_{i,j}:=\ker((\pi_{i,j})_*)$, which is the subgroup of all elements in $P_{i,j}$ having top left corner equal to the identity matrix of size $i$.

\smallskip
(b) For all $i=0,\ldots,n$ one has $\Injbar(U_i,G)=\Inj(U_i,G)$ which is a disjoint union of free right $\Aut(U_i)$-sets. Thus, for each $\chi_i\in\Irr_K(\Aut(U_i))$, one has $(U_i,\chi_i)\in\calE_G$. We will show that for $(U_i,\chi_i), (U_j,\chi_j)\in\calE_G$ one has  
\begin{equation}\label{eqn rel for el ab}
  e_{(U_i,\chi_i)}\cdotG KB^{\trl}(G,G)\cdotG e_{(U_j,\chi_j)}\neq \{0\} \iff \text{$i\le j$ and $\chi_j\mid\ind_{P_{i,j}}^{G_j} ( \infl_{(\pi_{i,j})_*}(\chi_i))$.}
\end{equation}
Here, for characters $\chi$ and $\chi'$ of a finite group, we write $\chi\mid\chi'$ if $\chi$ is a summand of $\chi'$. In fact, if $\Utilde_j$ (resp.~$\Utilde_i$) is a subgroup of $G$ isomorphic to $U_j$ (resp.~$U_i$) and if $\alpha\colon \Utilde_j\to \Utilde_i$ is an epimorphism then there exist isomorphisms $\lambda_j\colon U_j\myiso \Utilde_j$ and $\lambda_i\colon U_i\myiso \Utilde_i$ such that $\lambda_i^{-1}\alpha\lambda_j=\pi_{i,j}$. Applying Lemma~\ref{lem trivial Out_G}, we see that the condition in (\ref{eqn main cond}) holds if and only if the trivial character is a constituent of $\res^{\Aut(\Utilde_i)\times\Aut(\Utilde_j)}_{L_\alpha}((\chi_i)_{\Utilde_i}^*\times(\chi_j)_{\Utilde_j})$. Using the isomorphisms $\lambda_i$ and $\lambda_j$ this condition is equivalent to $\bigl(1,\res^{\Aut(U_i)\times\Aut(U_j)}_{L_{\lambda_i^{-1}\alpha\lambda_j}}(\chi_i^*\times \chi_j)\bigr)\neq 0$, and the claim follows again with Lemma~\ref{lem trivial Out_G}.

\smallskip
(c) This implies that each element $(U_i,\chi_i)\in\calE_G$ is equivalent (with respect to the equivalence relation $\approx$ in Corollary~\ref{cor general}) to an element of the form $(G,\chi)$ for some $\chi\in\Irr_K(\Aut(G))$, namely for any $\chi$ occurring as a constituent in $\ind_{P_{i,n}}^{G_n} ( \infl_{(\pi_{i,n})_*}(\chi_i))$. But we cannot, in general, determine which of the elements $(G,\chi)$, $\chi\in\Irr_K(\Aut(G))$, are equivalent. For $n=2$ one computes easily that $(G,\chi)\approx(G,\chi')$ if and only if $\chi=\chi'$ or if $\chi,\chi'\in\{1,\St\}$, where $\St$ denotes the Steinberg character.

\smallskip
(d) In this part we show that the set 
\begin{equation}\label{eqn unipotent}
  \{(U_i,\chi_i)\in\calE_G\mid \text{$i\in\{0,\ldots,n\}$, $\chi_i\in\Irr_K(G_i)$ unipotent}\}
\end{equation} 
is an equivalence class under the equivalence relation $\approx$ in Corollary~\ref{cor general}. To see this assume that $(U_i,\chi_i)$ and $(U_j,\chi_j)$ satisfy (\ref{eqn rel for el ab}). We will first show that $\chi_i$ is unipotent if and only if $\chi_j$ is unipotent. 

\smallskip
If $\chi_i$ is unipotent, i.e., $\chi_i\mid\ind_{B_i}^{G_i}(1)$, where $B_i$ denotes the subgroup of lower triangular matrices in $G_i$, then 
\begin{equation*}
  \chi_j\mid \ind_{P_{i,j}}^{G_j} \infl_{(\pi_{i,j})_*} (\chi_i) \mid 
  \ind_{P_{i,j}}^{G_j} \infl_{(\pi_{i,j})_*} \ind_{B_i}^{G_i}(1)\,.
\end{equation*}
But
\begin{align}\label{eqn inf ind}
    & \notag \infl_{(\pi_{i,j})_*} \ind_{B_i}^{G_i}(1) = 
    \ind_{(\pi_{i,j})_*^{-1}(B_i)}^{P_{i,j}} \infl_{(\pi_{i,j})_*\colon (\pi_{i,j})_*^{-1}(B_i)\to B_i} (1) \\
    = \ & \ind_{(\pi_{i,j})_*^{-1}(B_i)}^{P_{i,j}} (1) \mid 
    \ind_{(\pi_{i,j})_*^{-1}(B_i)}^{P_{i,j}} \ind_{B_j}^{(\pi_{i,j})_*^{-1}(B_i)} (1)=\ind_{B_j}^{P_{i,j}}(1)\,,
\end{align}
since $B_j\le (\pi_{i,j})_*^{-1}(B_i)$. Altogether we obtain $\chi_j\mid \ind_{P_{i,j}}^{G_j}\ind_{B_j}^{P_{i,j}}(1) = \ind_{B_j}^{G_j}(1)$ so that also $\chi_j$ is unipotent. 

\smallskip
Conversely, if $\chi_j$ is unipotent, then $\chi_j\mid \ind_{B_j}^{G_j}(1)$ and $\chi_i\mid \defl_{(\pi_{i,j})_*}\res^{G_j}_{P_{i,j}}(\chi_j)$ by (\ref{eqn rel for el ab}) and the obvious adjunctions. This implies
\begin{equation*}
  \chi_i\mid\defl_{(\pi_{i,j})_*}\res^{G_j}_{P_{i,j}} \ind_{B_j}^{G_j}(1)\,.
\end{equation*}
But, Mackey's decomposition formula yields
\begin{equation*}
  \res^{G_j}_{P_{i,j}} \ind_{B_j}^{G_j}(1) =
  \sum_{g} \ind_{P_{i,j}\cap\lexp{g}{B_j}}^{P_{i,j}}(1)
\end{equation*}
where $g$ runs over some subset of $G$. Thus, there exists $g\in G$ such that
\begin{equation*}
  \chi_i\mid \defl_{(\pi_{i,j})_*} \ind_{P_{i,j}\cap\lexp{g}{B_j}}^{P_{i,j}}(1) = 
  \isom_{\overline{(\pi_{i,j})_*}} \ind_{(P_{i,j}\cap \lexp{g}{B_j})Q_{i,j}/Q_{i,j}} ^{P_{i,j}/Q_{i,j}}(1)\,,
\end{equation*}
where $\overline{(\pi_{i,j})_*}\colon P_{i,j}/Q_{i,j}\myiso G_i$ denotes the isomorphism induced by $(\pi_{i,j})_*$ and $\isom_{\overline{(\pi_{i,j})_*}}$ denotes the corresponding isomorphism $R_K(P_{i,j}/Q_{i,j})\myiso R_K(G_i)$.  By Lemma~\ref{lem flag} below, the subgroup of $G_i$ corresponding to $(P_{i,j}\cap \lexp{g}{B_j})Q_{i,j}/Q_{i,j}$ under $\isom_{\overline{(\pi_{i,j})_*}}$ is conjugate to $B_i$. Thus, we obtain $\chi_i\mid \ind_{B_i}^{G_i}(1)$ and $\chi_i$ is unipotent.

Finally, we will show that each element in the set (\ref{eqn unipotent}) is equivalent to $(U_1,1)$. Note that, for $i=1,\ldots,n-1$, one has $(\pi_{i,i+1})_*^{-1}(B_i) = B_{i+1}$, so that (\ref{eqn inf ind}) becomes $\infl_{(\pi_{i,i+1})_*} \ind_{B_i}^{G_i}(1) = \ind_{B_{i+1}}^{P_{i,i+1}}(1)$. An easy induction argument now shows that
\begin{equation*}
  (\ind_{P_{i,i+1}}^{G_{i+1}}\infl_{(\pi_{i,i+1})_*})\circ(\ind_{P_{i-1,i}}^{G_i}\infl_{(\pi_{i-1,i})_*})\circ\cdots\circ
  (\ind_{P_{1,2}}^{G_2}\infl_{(\pi_{1,2})_*}) (1) = \ind_{B_{i+1}}^{G_{i+1}}(1)\,.
\end{equation*}
But this implies that, for each $i=1,\ldots,n-1$ and each unipotent character $\chi_{i+1}\in\Irr_K(G_{i+1})$, there exists a chain of unipotent characters $\chi_j\in\Irr_K(G_j)$, $j=2,\ldots,i$, such that $(U_1,1)\sim(U_2,\chi_2)\sim\cdots\sim(U_i,\chi_i)\sim(U_{i+1},\chi_{i+1})$. Now the proof of the claim is complete.
\end{example}

\begin{lemma}\label{lem flag}
Let $\calF\colon\{0\}=V_0\subset V_1\subset \cdots\subset V_n=V$ be a chain of subspaces in a vector space $V$ over a field $K$ with $\dim_K V_i=i$ for $i=0,\ldots,n$. Moreover, let $U$ be a subspace of $V$ and let $\overline{\cdot}\colon V\to V/U$, $v\mapsto \vbar$, denote the canonical epimorphism. Let $P$ denote the stabilizer of $U$ in $\Aut(V)$ and let $\pi\colon P\to \Aut(\Vbar)$ denote the epimorphism given by $(\pi(f))(\vbar)=\overline{f(v)}$ for $f\in P$ and $v\in V$. Then $\pi$ maps $\stab_P(\calF)$ onto $\stab_{\Aut(\Vbar)}(\calFbar)$, where $\calFbar$ denotes the chain $0=\Vbar_0\subseteq \Vbar_1\subseteq \cdots\subseteq \Vbar_n=\Vbar$ of subspaces of $\Vbar$.
\end{lemma}

\begin{proof}
It is straightforward to check that if $f\in\stab_P(\calF)$ then $\pi(f)\in\stab_{\Aut(\Vbar)}(\calFbar)$. 

Conversely, assume that $g\in\Aut(\Vbar)$ stabilizes $\calFbar$. By induction on $i$ we will construct a sequence $f_i\in\Aut(V_i)$, $i=0,\ldots,n$, such that 
\begin{gather}
  \notag
  \text{$f_i|_{V_{i-1}}=f_{i-1}$ for $i=1,\ldots,n$,}\\ 
  \label{eqn f_i property}
  \text{$f_i(V_i\cap U) = V_i\cap U$ for $i=0,\ldots,n$, and}\\ 
  \notag
  \text{$f_i(v)+U= g(v+U)$ for $v\in V_i$ and $i=0,\ldots,n$.}
\end{gather}
To this end, let $v_i$ be an element of $V_i$ not contained in $V_{i-1}$, for $i=1,\ldots,n$. Then $v_1,\ldots,v_i$ is a basis of $V_i$ for $i=0,\ldots, n$. We start with defining $f_0$ as the zero map and assume we have already defined $f_i$ satisfying the above properties. Then we define $f_{i+1}$ as the unique extension of $f_i$ with the property that $f_{i+1}(v_{i+1})= \alpha v_{i+1} + w$ for elements $0\neq\alpha\in K$ and $w\in V_i$ that will be determined by distinction of two cases. First note that the equations
\begin{gather*}
  \dim (V_i\cap U) + \dim (V_i+U) = \dim V_i + \dim U \\
  \dim (V_{i+1}\cap U) + \dim (V_{i+1}+U) = \dim V_{i+1} + \dim U
\end{gather*}
imply that either $V_i\cap U= V_{i+1}\cap U$ or $V_i + U = V_{i+1} + U$. In the first case one has $V_{i+1}+U= K v_{i+1}\oplus (V_i+U)$  and in the second case one has $V_{i+1}\cap U = K(v_{i+1}-v)\oplus (V_i\cap U)$, where $v\in V_i$ is any element satisfying $v_{i+1}+U = v+U$. In the first case we have $g(\vbar_{i+1}) = \alpha \vbar_{i+1} + \vbar$ for some $0\neq \alpha\in K$ and $v\in V_i$, since $g$ stabilizes $\calFbar$. In this case we define $f_{i+1}(v_{i+1}):=\alpha v_{i+1} +v$. In the second case let $v\in V_i$ be such that $v_{i+1}+U=v+U$ and set $f_{i+1}(v_{i+1}):=v_{i+1}-v+f_i(v)$. It is now straightforward to show that also $f_{i+1}$ satisfies the requirements in (\ref{eqn f_i property}) for the parameter $i$ replaced with $i+1$. Finally, the automorphism $f_n$ of $V$ has the property that $f\in P$, that $f$ stabilizes $V_i$ for all $i=0,\ldots, n$, and that $\pi(f)=g$.
\end{proof}


\end{document}